\documentclass[oneside,reqno,a4paper,11pt]{amsart}
\usepackage{amsmath,amsthm,amssymb,amscd,soul}

\usepackage{graphicx}
\usepackage{hyperref}
\usepackage{mathrsfs}
\usepackage{cite}
\usepackage{color}
\usepackage{ifpdf}
\usepackage{fancyhdr}
\usepackage{multirow}
\usepackage{array}
\usepackage{appendix}
\usepackage{longtable}
\usepackage{tikz}

\newtheorem{theorem}{Theorem}[section]
\newtheorem{lemma}[theorem]{Lemma}
\newtheorem{corollary}[theorem]{Corollary}
\newtheorem{proposition}[theorem]{Proposition}

\theoremstyle{definition}
\newtheorem{definition}[theorem]{Definition}

\theoremstyle{remark}
\newtheorem{remark}[theorem]{Remark}

\numberwithin{equation}{section}

\newcommand{\beq}{\begin{equation}}
	\newcommand{\eeq}{\end{equation}}
\newcommand{\ben}{\begin{eqnarray}}
	\newcommand{\een}{\end{eqnarray}}
\newcommand{\beno}{\begin{eqnarray*}}
	\newcommand{\eeno}{\end{eqnarray*}}

\numberwithin{equation}{section}
\subjclass[]{}
\keywords{}

\begin{document}
	
	\title{Regularity of the free boundary for a semilinear vector-valued minimization problem}
	\author{LILI DU$^{1,2}$}
	\author{YI ZHOU$^2$}
	\thanks{* This work is supported by National Nature Science Foundation of China Grant 11971331, 12125102, and Sichuan Youth Science and Technology Foundation 2021JDTD0024.}
	\thanks{$^1$ E-mail: dulili@szu.edu.cn \quad $^2$ E-mail: zhou\_yi@stu.scu.edu.cn}
	
	\maketitle
	
	\begin{center}
		$^1$College of Mathematical Sciences, Shenzhen University,
		
		Shenzhen 518061, P. R. China.
	\end{center}
	
	\begin{center}
		$^2$Department of Mathematics, Sichuan University,
		
		Chengdu 610064, P. R. China.
	\end{center}
	\begin{abstract}
		In this paper, we consider the following semilinear vector-valued minimization problem
		$$\min\left\{\int_{D}({|\nabla\mathbf{u}|}^2 + F(|\mathbf{u}|))dx: \ \ \mathbf{u}\in W^{1,2}(D; \mathbb{R}^m) \ \text{and} \ \mathbf{u}=\mathbf{g}\ \text{on} \ \partial D\right\}$$
		where $\mathbf{u}: D\to \mathbb{R}^m$ ($ m\geq 1$) is a vector-valued function, $D\subset \mathbb{R}^n$ ($n\geq 2$) is a bounded Lipschitz domain, $\mathbf{g}\in W^{1,2}(D; \mathbb{R}^m)$ is a given vector-valued function and $F:[0, \infty)\rightarrow \mathbb{R}$ is a given function. This minimization problem corresponds to the following semilinear elliptic system
		\begin{equation*}
			\Delta\mathbf{u}=\frac{1}{2}F'(|\mathbf{u}|)\cdot\frac{\mathbf{u}}{|\mathbf{u}|}\chi_{\{|\mathbf{u}|>0\}},
		\end{equation*}
		where $\chi_A$ denotes the characteristic function of the set A. The linear case that $F'\equiv 2$ was studied in the previous elegant work by Andersson, Shahgholian, Uraltseva and Weiss [Adv. Math 280, 2015], in which an epiperimetric inequality played a crucial role to indicate an energy decay estimate and the uniqueness of blow-up limit. However, this epiperimetric inequality cannot be directly applied to our case due to the more general non-degenerate and non-homogeneous term $F$ which leads to Weiss' boundary adjusted energy does not have scaling properties. Motivated by the linear case, when $F$ satisfies some assumptions, we establish successfully a new epiperimetric inequality, it can deal with term which is not scaling invariant in Weiss' boundary adjusted energy. As an application of this new epiperimetric inequality, we conclude that the free boundary $D\cap \partial\{|\mathbf{u}|>0\}$ is a locally $C^{1,\beta}$ surface near the regular points for some $\beta\in (0,1)$.

		\noindent{Keyword: } Free boundary; Vector-valued; Elliptic system; Regularity; Obstacle problem.
	\end{abstract}
	\maketitle

	\section{Introduction}

	In this paper,  we study the  semilinear vector-valued minimization problem
	\begin{equation}\label{eq1.0}\min\left\{\int_{D}{|\nabla\mathbf{u}|}^2+F(|\mathbf{u}|)dx: \ \ \mathbf{u}\in W^{1,2}(D; \mathbb{R}^m) \ \text{and} \ \mathbf{u}=\mathbf{g}\ \text{on} \ \partial D\right\},
	\end{equation}
	where $\mathbf{u}: D\to \mathbb{R}^m$ ($ m\geq 1$) is a vector-valued function, $\nabla=(D_{x_1},..., D_{x_n})$ denotes the standard gradient operator in $\mathbb{R}^n$, $D\subset \mathbb{R}^n$ ($n\geq 2$) is a bounded Lipschitz domain, $\mathbf{g}\in W^{1,2}(D; \mathbb{R}^m)$ is a given vector-valued function and $F :[0, \infty)\rightarrow \mathbb{R}$ is a given function. It is easy to show (see Appendix A) that this minimization problem has a unique minimizer $\mathbf{u} \in W^{2,p}(D;\mathbb{R}^m)$,  solving the following semilinear elliptic system
	\begin{equation}\label{eq1.1}
		\Delta\mathbf{u}=f(|\mathbf{u}|)\cdot\frac{\mathbf{u}}{|\mathbf{u}|}\chi_{\{|\mathbf{u}|>0\}},
	\end{equation}
	in $D$ with the given boundary data $\mathbf{g}\in W^{1,2}(D,\mathbb{R}^m)$, where $f(s)=\frac{1}{2}F'(s)$ and $\chi_A$ denotes the characteristic function of the set A.
	
	Our main purpose in this paper is to study the regularity of the free boundary $D\cap\partial\{ |\mathbf{u}|>0\}$. For simplicity, we shall denote $D\cap\partial\{ |\mathbf{u}|>0\}$ by $\Gamma(\mathbf{u}):=\partial\{ |\mathbf{u}|>0\}$.  For further analysis of the free boundary regularity, $\Gamma(\mathbf{u})$ can be divided into two parts
	$$\Gamma(\mathbf{u})= \Gamma_0(\mathbf{u})\cup \Gamma_1 (\mathbf{u}),$$
	where $\Gamma_0(\mathbf{u}):=\Gamma(\mathbf{u})\cap \{x\in D \mid |\nabla\mathbf{u}(x)|=0\}$ and $\Gamma_1(\mathbf{u}):=\Gamma(\mathbf{u})\cap \{x\in D \mid |\nabla\mathbf{u}(x)|\neq0\}$ represent the  degenerate part and the non-degenerate part, respectively. It should be noted that $\Gamma_1(\mathbf{u})$ is a locally $C^{1,\alpha}$-surface, as a direct result of the implicit function theorem, so that we are more concerned with the degenerate part where the gradient vanishes.
	
	\subsection{Background of the problem}\
	
	It is worth pointing out that the elliptic system \eqref{eq1.1} can be used to describe many important models in the areas of applied mathematics such as mathematical biology, etc.  Particularly, when $m=2$, a cooperative system can be given by the corresponding reaction-diffusion system
	\begin{align*}
		\begin{split}
			\left\{
			\begin{array}{lr}
				u_t-\Delta u = -f\left(\sqrt{u^2+v^2}\right)\displaystyle\frac{u}{\sqrt{u^2+v^2}},\\
				\ \\	
				v_t-\Delta v = -f\left(\sqrt{u^2+v^2}\right)\displaystyle\frac{v}{\sqrt{u^2+v^2}},\\
			\end{array}
			\right.
		\end{split}
	\end{align*}
	i.e.
	$$\mathbf{w}_t-\Delta \mathbf{w} = -f\left(|\mathbf{w}|\right)\displaystyle\frac{\mathbf{w}}{|\mathbf{w}|}, \ \text{for}\  |\mathbf{w}|>0,\quad\text{with}\quad  \mathbf{w}:=(u,v),$$
	which means that an increase in each species/reactant will decelerate the extinction/reaction of all species/reactants (see \cite{m02}). For example, in recent interesting works, J. Andersson et al.  \cite{asu} and G. Aleksanyan et al. \cite{afsw} dealt with the corresponding reaction-diffusion systems connected with $f(s)=1$ and $f(s)=s^{q} (0\leq q<1)$ respectively. Here, we would like to emphasize that the behavior of the solution for non-degenerate case $f(s)\geq c_0>0$ and degenerate case $f(0)=0$ are quite different, such as the optimal regularity of the solution, the decay rate near the free boundary and so on. Indeed, for the degenerate case, it can be shown that the decay rate of the solution is faster than quadratic order at the free boundary points. However, for the non-degenerate case, this assumption prevents degeneracy of solution, so we shall prove that the solution has optimal decay $r^2$, and not faster. Henceforth, in this paper, we will consider the more general non-degenerate case (see assumption \eqref{eq1.3}) and show the regularity of the free boundary of the vector-valued problem \eqref{eq1.0}.
	
	Due to the importance of such models mentioned above, the elliptic system \eqref{eq1.1}, especially regularity of its free boundary, has attracted much attention from many mathematicians (see \cite{afsw, asu, fsw} and the references therein).  For the scalar case ($m=1$),  the elliptic system \eqref{eq1.1} is reduced to the following obstacle-problem-like equation
	$$\Delta u= f(u) \chi_{\{u>0\}} - f(-u)\chi_{\{u<0\}},$$
	which corresponds to a class of two-phase free boundary problems. For the linear case, $f(s)$ is a positive constant, N. Uraltseva et al. \cite{u} derived the boundedness of $D^2u$ of solutions using the Alt--Caffarelli--Friedman monotonicity formula  (see \cite{acf2,ca3}). A few years later, H. Shahgholian et al. in  \cite{suw04} have given a complete characterization of all global two-phase solutions with quadratic growth both at $0$ and infinity, using the Alt--Caffarelli--Friedman monotonicity formula and the so-called Weiss-type monotonicity formula which was established in the ground-breaking paper \cite{w99} by G. S. Weiss for the classical obstacle problem. In \cite{suw07}, H. Shahgholian et al. considered the following case
	$$\Delta u= \lambda_{+}(x) \chi_{\{u>0\}} - \lambda_{-}(x) \chi_{\{u<0\}},$$
	where $\lambda_{\pm}(x) $ are both positive Lipschitz functions. They proved that the free boundary is in a neighbourhood of each "branch point" the union of two $C^1$- graphs.  In 2017, M. Fotouhi, H. Shahgholian \cite{fs} studied the sublinear case
	$$\Delta u= \lambda_{+}(x) (u^{+})^{q} - \lambda_{-}(x) (u^{-})^{q},$$
	where $0<q<1$, $u^{\pm}=\max\{0,\pm u\}$ and $\lambda_{\pm}(x) $ are both positive Lipschitz functions, and proved that if a solution is close to one-dimensional solution in a small ball, the free boundary can be represented locally as two $C^1$-regular graphs $\partial\{u>0\}$ and $\partial\{u<0\}$, tangential to each other.
	
	As for the vector-valued case ($m>1$), the situation is much more complicated than that of the scalar case. One of the main difficulties is that a corresponding version of the Alt--Caffarelli--Friedman monotonicity formula seems to be unavailable in the vector-valued problem. Motivated by the result of W. H. Fleming \cite{f}, he considered	the problem of finding a minimal surface with given the same boundary (also see \cite{re}). Then G. S. Weiss first applied the new epiperimetric-type approach to study the regularity of the free boundary of classical obstacle problem \cite{w99}. And it has been extended to more sophisticated problems, involving for example vector-valued free boundary problems \cite{afsw, asu, fsw}. One special interest for us is the important recent breakthrough work \cite{asu}, which considered the linear case ($f(s)=1$)
	$$\Delta \mathbf{u}= \frac{\mathbf{u}}{|\mathbf{u}|} \chi_{\{|\mathbf{u}|>0\}}, $$
	and proved that the free boundary $\partial\{|\mathbf{u}|>0\}$ is  a $C^{1,\beta}$-surface for some $\beta>0$ in an open neighbourhood of the set of regular free boundary points (we say $x\in\Gamma(\mathbf{u})$ is a {\it regular free boundary point} if at least exist one blow up limit of $\mathbf{u}$ at $x$ is a half-space solution). For this case, D. D. Silva et al. \cite{sjs} considered the almost minimizers for \eqref{eq1.0}, for which they obtained the regularity of both at almost minimizers and the regular part of the free boundary.
	Later, M. Fotouhi et al. \cite{fsw}  and  D. D. Silva et al. \cite{sjs-1} considered a sublinear vector case
	\begin{equation}\label{f}
		\Delta\mathbf{u}=f(|\mathbf{u}|)\frac{\mathbf{u}}{|\mathbf{u}|}\chi_{\{|\mathbf{u}|>0\}},
	\end{equation}
	for $f(|\mathbf{u}|)=\lambda_{+}|\mathbf{u}^{+}|^{q}+\lambda_{-}|\mathbf{u}^{-}|^{q}$,
	where $\lambda_{\pm}(x) $ are both positive H\"{o}lder continuous functions, and $0<q<1$, with $\mathbf{u}: B_1\subset \mathbb{R}^n\to \mathbb{R}^m, n\geq 2, m\geq 1$, and $\mathbf{u}^{\pm}=(u_1^{\pm},...,u_m^{\pm})$, $u_i^{\pm}=\max\{0,\pm u_i\}$. Using the approach of epiperimetric inequality, for the minimizer and almost minimizers, respectively, they obtained that the free boundary is a locally $C^{1,\beta}$ surface near regular points with some $\beta>0$. As we mentioned before, in \cite{fsw,w00} the problem \eqref{f} for the case $f(s)=s^q \left( \text{for}\ q\in(0,1)\right)$ belongs to the degenerate case ($f(0)=0$), then the optimal regularity of the solution should be $C^{[\kappa],\kappa-[\kappa]}$, where $\kappa=\frac{2}{1-q}$, and the decay rate of solution is $r^\kappa$ near the free boundary. However, in this paper, we will investigate the non-degenerate case $f(s)\geq c_0>0$, and will show the optimal regularity is $C^{1,\alpha}$ and the decay rate is in fact $r^2$ near the free boundary. Moreover, based on the basic properties of the solution and the free boundary, we can show that the free boundary is indeed $C^{1,\beta}$ near the regular points.

	\subsection{Basic assumptions and definitions} \
	
	In this paper, we need to restrict the nonlinear function $F(s)$ to a reasonable class. More precisely, we assume that $F(s)\in C^{2, \alpha}\left([0,\infty)\right)$ for some $\alpha\in (0, 1)$ and satisfies
	\begin{equation}\label{eq1.3}
		F(0)= 0, \  c_0 \leq f(s):=\frac{1}{2}F'(s)\leq C_0 \quad \text{and} \quad 0\leq F''(s) \leq C_0\quad \text{for}\quad s\geq 0,
	\end{equation}
	where $c_0, C_0$ are two given positive constants.

	\begin{remark}\label{rem1.1}
		Due to the convexity of $F$,  it is easy to see
		$$2f(0)s=F'(0)s\leq F(s)\leq F'(s)s=2f(s)s, \quad\text{for}\quad s\geq 0.$$
	\end{remark}
	\begin{proposition}\label{rem1.2}\
		With the assumptions above, the system \eqref{eq1.1} has a unique minimizer $\mathbf{u}\in W^{2,p}(D;\mathbb{R}^m), \ \text{for any}\ p\in [1,\infty)$.
	\end{proposition}
	\begin{proof}
		For a complete proof of this fact see the Appendix A.
	\end{proof}

	\begin{remark}
		Note that for the scalar case corresponding to the classical obstacle problem (see \cite{psu}), the optimal regularity of the solution is $C^{1,1}(D)$. Then for the vector case, we notice that J. Andersson et al. have proposed the open problem whether $\mathbf{u}$ belongs to $W^{2, \infty}(D;\mathbb{R}^m)$ or not of the linear case in \cite{asu}. However, for our more general vector case, the optimal regularity of $\mathbf{u}$ can not achieve $W^{2, \infty}(D;\mathbb{R}^m)$. Even for the scalar case, it is already realized that there is a counterexample in \cite{hl}. Specifically, in the $n-$dimensional ball $B_R=B_R(0)$ of radius $R<1$, we consider the following equation
		$$ \Delta u=\frac{x_2^2-x_1^2}{2|x|^2}\left\{\frac{n+2}{(-\log|x|)^{\frac{1}{2}}}+\frac{1}{2(-\log|x|)^{\frac{3}{2}}}\right\},$$
		the solution $u(x)\in C(\bar{B}_R)\cap C^{\infty}(\bar{B}_R \backslash\{0\})$ but $u(x)$ is not in $W^{2,\infty}(B_R)$ that has been proved in \cite{hl}.
	\end{remark}

	In the following, some important notions will be given before stating our main results for the convenience. Firstly, we denote the energy of \eqref{eq1.0} in the domain $D$ and the unit ball by
	$$E(\mathbf{u}):=\int_{D}{|\nabla\mathbf{u}|}^2+F(|\mathbf{u}|)dx,$$
	and
	$$E(\mathbf{u},x_0,1):=\int_{B_1(x_0)}{|\nabla\mathbf{u}|}^2+F(|\mathbf{u}|)dx,$$
	respectively. Now we give the so-called {\it Weiss' boundary adjusted energy} and {\it energy density}  which are introduced in \cite{w99}.
	
	\begin{definition} {\it(Weiss' boundary adjusted energy) }\label{weissenergy}
		Let $\mathbf{u}$ be a solution of \eqref{eq1.1} in $B_{r_0}(x^0)$. Then one can define
		$$ W(\mathbf{u},x^0,r):=\frac{1}{r^{n+2}}\int_{B_r(x^0)}|\nabla\mathbf{u}|^2+F(|\mathbf{u}|)dx -\frac{2}{r^{n+3}}\int_{\partial B_r(x^0)}|\mathbf{u}|^2 d\mathcal{H}^{n-1}, $$
		for $0<r\leq r_0$, where $B_r(x^0)$ denotes the open ball in $\mathbb{R}^n$ centered at $x^0$ and radius $r$, and $\mathcal{H}^k$ denotes $k-$dimensional Hausdorff-measure.
	\end{definition}
	
	\begin{remark}
		Let $\mathbf{u}(x)$ be the solution of \eqref{eq1.1} in $B_{r_0}(x^0)$, based on the Non-degeneracy (Proposition \ref{non-degeneracy}) and Growth estimate (Proposition \ref{growth} ), we define the following rescaled solution at the free boundary point $x^0$,
		$$\mathbf{u}_{x^0,r}:=\displaystyle\frac{\mathbf{u}(x^0+rx)}{r^2}.$$
	\end{remark}
		After scaling, Weiss' boundary adjusted energy can be written as follows,
		$$W(\mathbf{u},x^0,r)=\int_{B_1(0)}|\nabla \mathbf{u}_{x^0,r} |^2+\frac{1}{r^2}F(r^2 |\mathbf{u}_{x^0,r}|)dx-2\int_{\partial B_1}|\mathbf{u}_{x^0,r}|^2d \mathcal{H}^{n-1}.$$
		For simplicity, let
		$$H(\mathbf{v},r):=\int_{B_1}|\nabla \mathbf{v}|^2+\frac{1}{r^2}F(r^2 |\mathbf{v}|)dx-2\int_{\partial B_1}|\mathbf{v}|^2d \mathcal{H}^{n-1},$$
		where $B_1:=B_1(0).$
		According to \cite[Theorem 1.1]{l22} and the monotonicity of Weiss's energy functional (Lemma \ref{monotonicity}), it follows that  the limit $\displaystyle\lim_{r\rightarrow 0+} H(\mathbf{u}_{x^0,r},r)$ exists  and the convexity of $F$ (recalling Remark \ref{rem1.1}) implies that
		$$\displaystyle\lim_{r\rightarrow 0+} H(\mathbf{u}_{x^0,r},r)=\int_{B_1}|\nabla \mathbf{u}_0|^2+2f(0)|\mathbf{u}_0|dx-2\int_{\partial B_1}|\mathbf{u}_0|^2d \mathcal{H}^{n-1},$$
		where $\mathbf{u}_0$ denotes the blow-up limit of $\mathbf{u}_{x^0,r}$.
		We also denote
		$$M(\mathbf{v}):=\int_{B_1}|\nabla \mathbf{v}|^2+2f(0)|\mathbf{v}|dx-2\int_{\partial B_1}|\mathbf{v}|^2d \mathcal{H}^{n-1}, \quad\text{for any}\ \mathbf{v}\in \mathbb{R}^m,$$
		i.e.
		$$\displaystyle\lim_{r\rightarrow 0+} H(\mathbf{u}_{x^0,r},r)=M(\mathbf{u}_0).$$
		The limit $\displaystyle\lim_{r\rightarrow 0+} H(\mathbf{u}_{x^0,r},r)$ is called the energy density of $\mathbf{u}$ at $x^0$.

	To study the regularity of free boundary, we will define {\it half space solutions} and {\it regular free boundary points}. After scaling, the equation \eqref{eq1.1} leads to the blow-up limit $\mathbf{u}_0$ satisfying the following equation (the detailed information can be found in \eqref{eq3.7})
	
	\begin{equation*}
		\Delta \mathbf{u}_0 = f(0)\frac{\mathbf{u}_0}{|\mathbf{u}_0|}\chi_{\{|\mathbf{u}_0|>0\}},
	\end{equation*}
	
	thus, we provide the definition of a half-space solution.
	\begin{definition} {\it(Half space solutions)}
		The set of half space solutions is given by
		\begin{equation}\label{eq1.4}
			\mathbb{H}:=\left\{\frac{f(0)\text{max}(x\cdot\nu,0)^2}{2}\mathbf{e}: \nu \ \text{unit\ vector\ of }\ \mathbb{R}^n, \mathbf{e} \ \text{unit\ vector\ of} \ \mathbb{R}^m\right\}.
		\end{equation}
	\end{definition}
	\begin{remark} \label{rem1.7}
		As a matter of fact, the energy density value  $$M\Big(\frac{f(0)\text{max}(x\cdot\nu,0)^2}{2}\mathbf{e}\Big)$$ does not depend on the choices of $\nu$ and $\mathbf{e}$, denoted as $\displaystyle\frac{\alpha_n}{2}$, where $\alpha_n$ is a constant, that will be shown in Section 4.
	\end{remark}

	Before, definition the set of free boundary points at which at least one blow-up limit coincides with a half-plane solution, now, we will turn out to be the fact that the half-plane solutions take a lower energy level that any other homogeneous solution of degree 2.

	\begin{definition}{\it(Regular free boundary points)}\label{regular}
		A point $x^0$ is called a regular free boundary point for $\mathbf{u}$ provided that
		$$x^0\in\Gamma_0(\mathbf{u}):=\partial\{ |\mathbf{u}|>0\} \cap \{\nabla\mathbf{u}(x^0)=0\}\qquad \text{and} \qquad \displaystyle\lim_{r\rightarrow 0}W(\mathbf{u},x^0,r)=\displaystyle\frac{\alpha_n}{2},$$
		here we denote by $\mathcal{R}_{\mathbf{u}}$ the set of all {\it regular free boundary points} of $\mathbf{u}$ in $B_1$.
	\end{definition}

	\begin{remark}
		In this work, to investigate the regularity of free boundary, we mainly consider the regular set $\mathcal{R}_{\mathbf{u}}$. As for the singular set of free boundary points, the analysis of regularity is still unknown up to now and this would be our future research direction.
	\end{remark}
	
	\subsection{Main results and plan of this paper}\
	
	Our main result concerning the regularity of the free boundary $\Gamma(\mathbf{u})$ is presented in the following theorem.
	\begin{theorem}{(Regularity)}\label{regularity}
		The free boundary $\Gamma(\mathbf{u})$ is in an open neighbourhood of the regular points set $\mathcal{R}_{\mathbf{u}}$ locally a $C^{1,\beta}$-surface. Here $\beta=\frac{(n+2)\kappa}{2(1-\kappa)}(1+\frac{(n+2)\kappa}{2(1-\kappa)})^{-1}$, for some $\kappa\in(0,1)$.
	\end{theorem}
	
	\begin{remark}
		The parameter $\kappa$ is given by the {\it epiperimetric inequality} in Theorem \ref{epiperimetric}, and it is easy to show that $\beta\in (0,1)$.
	\end{remark}
	
	\begin{remark} The proof of the regularity of free boundary (Theorem \ref{regularity}) main relies on the growth estimate and the uniqueness of blow-up limit. The epiperimetric inequality (contraction of energy) is the key to prove the uniqueness of blow-up limit.
	\end{remark}
	
	\begin{remark}
	
	It is noteworthy that in our more general non-degenerate case, the decay rate of the solution near the free boundary is $r^2$. However, it is possible that $F(u)$ does not have scaling properties. This is in stark contrast to the situations presented in \cite{afsw,asu,fsw}. Consequently, in our research, overcoming some essential difficulties after scaling, particularly in the establishment of epiperimetric inequalities, becomes imperative.
	\end{remark}
	
	The main goal of this paper is devoted to the proof of Theorem \ref{regularity} and  our approach  is  inspired by the celebrated work \cite{asu}, depending on the establishment of the {\it epiperimetric inequality}. The remaining part of our paper is organized as follows.
	
	Section 2 is devoted to establish some important properties of $\mathbf{u}$. More precisely, Proposition \ref{bounded} gives $L^{\infty}$ estimates for $\mathbf{u}$ and $\nabla\mathbf{u}$, and Proposition \ref{non-degeneracy} describes the non-degeneracy property of the solution $\mathbf{u}$. Furthermore, Proposition \ref{pro2.3} tells us that if a solution $\mathbf{u}$ is a small perturbation of a half space solution $\mathbf{h}$, then the support of $\mathbf{u}$ is also a small perturbation of that of  $\mathbf{h}$.  In Section 3, based on Proposition \ref{non-degeneracy}, the Weiss-type monotonicity formula (Lemma \ref{monotonicity}) will be introduced, and it will play a crucial role in our analysis for the regularity of free boundary. Due to Lemma \ref{monotonicity}, one can deduce Lemma \ref{property}, which shows that all blow-up limits have to be 2-homogeneous functions.  Furthermore, Proposition \ref{growth} has been established quadratic growth estimate for $\mathbf{u}$, which will be well-applied in the proof of the  regularity of free boundary. In section 4, we mainly study the properties of the 2-homogeneous global solutions. Proposition \ref{prop4.1} shows that when the support of 2-homogeneous global solutions lies in a small perturbation of half space, then it belongs to $\mathbb{H}$ defined in \eqref{eq1.4}. Based on Proposition \ref{pro2.3} and Proposition \ref{prop4.1}, one can establish Corollary \ref{isolated}, which gives that the half space solutions are in $L^1(B_1(0); \mathbb{R}^m)$-topology isolated with the class of 2-homogeneous global solutions. Thanks to Proposition \ref{prop4.4} and Corollary \ref{coro4.5}, one may know how to distinguish half space solutions from homogeneous solutions by calculating the value of the energy density $M(\mathbf{u})$. Section 5 is devoted to show the epiperimetric inequality (Lemma \ref{epiperimetric}), which implies an energy decay estimate and uniqueness of blow-up limit (Proposition \ref{uniqueness}). These results provide the basis for our further analysis of the regularity of free boundary. Section 7 is to verify the assumption of the energy decay estimate (Proposition \ref{uniqueness}) uniformly in an open neighborhood of a regular free boundary point (Lemma \ref{regularity1}), then using the key energy decay estimate (Lemma \ref{regularity1}) we prove Theorem \ref{regularity} via a standard iteration process as G. S. Weiss in \cite{asu}.

	\subsection{Notations}\
	
	For convenience, we list some notations that will be used in our paper.
	
	$\bullet$  $B_r^+(x^0):=\{x\in B_r(x^0): x_n>x^0_n\}$ and $\mathbf{e}^i$ is denoted the $i-$th unit vector in $\mathbb{R}^m$;

	$\bullet$ for any set $A\subset \mathbb{R}^n$,  we denote $A^0$ and  $|A|$ to be the interior of $A$ and the $n-$dimensional Lebesgue measure of $A$ when $A$ is Lebesgue measurable, respectively;

	$\bullet$ denote $\nu$ to be the topological outward normal of the boundary of a given set, and $\nabla_{\theta}f:=\nabla f- (\nabla f\cdot \nu)\nu  $ to be the surface derivative of a given function $f$;

	$\bullet$  for any $\mathbf{v}=(v^1, \cdots, v^m)$, $\mathbf{w}=(w^1, \cdots, w^m)$: $\mathbb{R}^n\rightarrow \mathbb{R}^m$, and $\xi\in\mathbb{R}^n$, we denote $\nabla \mathbf{v}=[\partial_i v^j]_{1\leq i\leq n, 1\leq j\leq m}$, $|\nabla \mathbf{v}|^2=\sum_{i=1}^m|\nabla v^i|^2$, $\nabla \mathbf{v}\cdot \xi=(\nabla v^1\cdot \xi,\cdots, \nabla v^m\cdot \xi)$.

	$\bullet$ $o(t)$ represents the higher order infinitesimal than $t$, i.e., $\displaystyle\lim_{t\to 0}\displaystyle\frac{o(t)}{t}=0$.

	\vspace{25pt}

	\section{$L^\infty$ estimates and Non-degeneracy }
	
	In this section, our goal is to show the $L^\infty$ estimates and non-degeneracy of solution to \eqref{eq1.1},  which will be heavily used  in following sections.
	
	\begin{proposition}{($L^\infty$ estimates)}\label{bounded}
		Let $\mathbf{u}$ be a solution of the system \eqref{eq1.1} in $B_1(0)$. Then,
		\begin{equation}\label{eq2.0}
			\underset{B_{3/4}}{\sup} |\mathbf{u}|+\underset{B_{3/4}}{\sup} |\nabla\mathbf{u}|\leq C \left(\|\mathbf{u}\|_{L^1(B_1;\mathbb{R}^m)}+1\right),
		\end{equation}where the constant $C$ depends only on $n, m, c_0 $ and $C_0$.
	\end{proposition}
	\begin{proof}The proof can be carried out by standard elliptic regularity theory. More precisely, for the semilinear elliptic system, since the right hand side of the system is bounded, then using the local maximum principle for strong solution (see Remark \ref{rem1.2} (i) and  \cite[Theorem 9.20]{gt}), one can get
		$$\underset{B_{3/4}}{\sup} |\mathbf{u}|\leq C \left(\|\mathbf{u}\|_{L^1(B_1;\mathbb{R}^m)}+1\right),$$
		with $C=C(n, m, C_0)$>0.
		
		As for the $L^\infty$ estimate for the gradient $|\nabla\mathbf{u}|$, one may use the standard $W^{2,p}$ estimates for strong solution (see (9.40) in the proof of \cite[Theorem 9. 11]{gt}) to establish, for large $p>n$,
		$$\|D^2  \mathbf{u}\|_{L^p(B_{3/4};\mathbb{R}^m)}\leq C( C_0+ \underset{B_{5/6}}{\sup} |\mathbf{u}| ) \leq C \left(\|\mathbf{u}\|_{L^1(B_1;\mathbb{R}^m)}+1\right),$$
		where $|D^2\mathbf{u}|^p:=\sum_{i=1}^{m}|D^2 u^i|^p$, which together with the interpolation inequality (see (9.39) in the proof of \cite[Theorem 9. 11]{gt}) implies $\nabla\mathbf{u}\in W^{1, p}(B_{3/4} ;\mathbb{R}^m)$ and
		$$\|\nabla  \mathbf{u}\|_{W^{1, p}(B_{3/4};\mathbb{R}^m)}\leq  C \left(\|\mathbf{u}\|_{L^1(B_1;\mathbb{R}^m)}+1\right).$$
		Finally, from Sobolev imbedding theorem, $W^{1, p}(B_{3/4};\mathbb{R}^m) \hookrightarrow C^0(B_{3/4};\mathbb{R}^m)$ for large $p>n$ (see (7.30) in \cite{gt}), it follows that
		$$\underset{B_{3/4}}{\sup} |\nabla\mathbf{u}|\leq C \left(\|\mathbf{u}\|_{L^1(B_1;\mathbb{R}^m)}+1\right).$$
	\end{proof}

	\begin{proposition}{(Non-degeneracy)}\label{non-degeneracy}
		Let $\mathbf{u}$ be a solution of \eqref{eq1.1} in D. If $x^0 \in \overline{\{|\mathbf{u}|>0\}}$ and $ B_r(x^0)\subset D$, then
		\begin{equation}\label{eq2.2}
			\underset{B_r(x^0)}{\sup} |\mathbf{u}|\geq \frac{f(0)}{2n}r^2.
		\end{equation}
	\end{proposition}
	\begin{proof}
		Let $U(x):=|\mathbf{u}|$, then a direct computation gives that
		\begin{align}\label{2.2}
			\Delta U(x)&=-\frac{(\mathbf{u}\cdot\nabla\mathbf{u})^2}{|\mathbf{u}|^3}+\frac{|\nabla\mathbf{u}|^2}{|\mathbf{u}|}+\frac{\mathbf{u}\cdot \Delta \mathbf{u}}{|\mathbf{u}|}\nonumber\\
			&=\frac{A}{U}+{f( |\mathbf{u}|)}\ \  \text{in}\ \  \{ |\mathbf{u}|>0 \},
		\end{align}
		where $A= |\nabla\mathbf{u}|^2 -|\nabla U|^2\geq 0$. Define
		$$v(x):=U(x)-U(x^0)-\frac{f(0)}{2n}|x-x^0|^2.$$
		Then it infers that
		$$\Delta v(x)=\frac{A}{U}+f( |\mathbf{u}|)-f(0)\geq \frac{A}{U}\geq 0 \quad \text{in}\quad \{ |\mathbf{u}|>0 \}, $$
		namely, $v(x)$ is subharmonic in the connected component of $B_r(x^0)\cap \{|{\mathbf{u}}|>0\}$.

		Here we have two cases,
		
		$\bullet$ $x^0\in \{|{\mathbf{u}}|>0\}$. Assume the estimate \eqref{eq2.2} does not hold true, then \begin{align}\label{eq2.3}\underset{B_r(x^0)}{\sup} |\mathbf{u}|&< \frac{f(0)}{2n}r^2.\end{align}
		Thus this together with the definition of $v$, one may see
		$$v(x)<\frac{f(0)}{2n}r^2-U(x^0)-\frac{f(0)}{2n}r^2 < 0, \ \forall x\in \partial B_r(x^0)\cap \{|{\mathbf{u}}|>0\}.$$
		
		If $B_r(x^0)\cap \partial\{|{\mathbf{u}}|>0\}\neq\varnothing$,  we conclude that $v(x)=-U(x^0)-\frac{f(0)}{2n}|x-x^0|^2<0$, for all $x\in B_r(x^0)\cap \partial\{|{\mathbf{u}}|>0\}.$
		
		If $B_r(x^0)\cap \partial\{|{\mathbf{u}}|>0\}=\varnothing$,  then one only needs to consider $x\in \partial B_r(x^0)\cap \{|{\mathbf{u}}|>0\}$.
		
		Hence $v(x)<0$, for any $x\in \partial(B_r(x^0)\cap \{|{\mathbf{u}}|>0\})$. Since $\Delta v \geq 0$ in $B_r(x^0)\cap \{|{\mathbf{u}}|>0\}$, then the maximum principle (see chapter 2 in \cite{gt}) leads to the conclusion that
		$$v<0 \ \text{in}\ B_r(x^0)\cap \{|{\mathbf{u}}|>0\},$$
		which  contradicts with $v(x^0)=0$. Therefore this Proposition holds true when $x^0\in \{|{\mathbf{u}}|>0\}$.
		
		$\bullet$ $x^0\in \partial\{|{\mathbf{u}}|>0\}$.
		
		For this case, we may choose a sequence $\{x^k\} \subset \{|{\mathbf{u}}|>0\}$ such that $x^k \rightarrow x^0$ as $k\rightarrow \infty$. Then for any $k\geq 1$, it is well-known that
		$$\underset{B_r(x^k)}{\sup} |\mathbf{u}|\geq \frac{f(0)}{2n}r^2.$$ Thus one can get
		$$\underset{B_r(x^0)}{\sup} |\mathbf{u}|\geq \frac{f(0)}{2n}r^2,$$
		according to the continuity argument since $\mathbf{u}\in C^{1, \alpha}(D,\mathbb{R}^m)$ for any $\alpha\in (0, 1)$.
	\end{proof}

	From Proposition \ref{bounded} and \ref{non-degeneracy}, it infers
	\begin{proposition}\label{pro2.3}
		Let $\mathbf{u}$ be a solution of \eqref{eq1.1} in $B_{1}(0)$ such that $\|\mathbf{u}-\mathbf{h}\|_{L^1({B_1(0);\mathbb{R}^m}})\leq \epsilon<1,$ where $\mathbf{h}(x)=\frac{f(0)\text{max}(x_n,0)^2}{2}\mathbf{e}^1$. Then,
		$$\left(B_{\frac{1}{2}}(0)\cap supp \   \mathbf{u}\right) \subset \{x_n>-C\epsilon^{\frac{1}{2n+2}}\},$$
		where constant $C=C(n,m).$(See Fig.1)
		\begin{figure}[!h]
			\includegraphics[width=120mm]{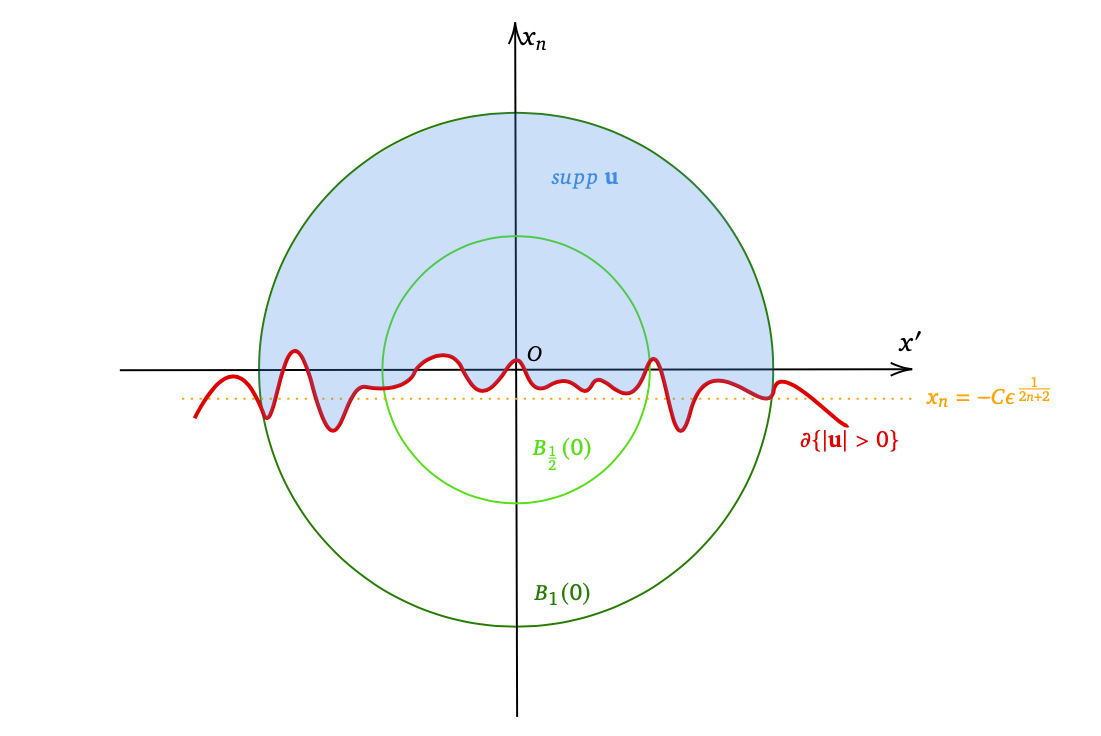}
			\caption{The support of $\mathbf{u} $}
		\end{figure}
		
	\end{proposition}
	
	We notice that $\mathbf{u}$ has similar non-degeneracy as the one in \cite{asu}, it follows that using the non-degeneracy (Proposition \ref{non-degeneracy}) and Harnack inequality (see \cite[Theorem 8.18]{gt}) as employed in \cite[Proposition 2 ]{asu} can be completed the proof of this proposition, and we omit it for simplicity.

	\begin{remark} About this proposition, we can understand that if a solution $\mathbf{u}$ is a small perturbation of a half space solution $\mathbf{h}$, then the support of  $\mathbf{u}$ is also a small perturbation of $\mathbf{h}$ in a small ball (See Fig.1).
		
	\end{remark}

	\vspace{25pt}
	
	\section{Weiss-type monotonicity formula}
	
	In this section, we will introduce a useful tool, the so-called Weiss-type monotonicity formula, which was introduced firstly in the ground breaking work \cite{w98} by G. S. Weiss to deal with the obstacle problem (see also \cite{w99,w01}).
	\begin{lemma} {(Weiss-type monotonicity formula) }\label{monotonicity}
		Let $\mathbf{u}$ be a solution of \eqref{eq1.1} in $B_{r_0}(x^0)$ and the definition of $W(\mathbf{u},x^0,r)$ is given in  Definition \ref{weissenergy}, then there holds
		\begin{eqnarray*}\frac{dW(\mathbf{u},x^0,r)}{dr} &=&\frac{2}{r^{n+2}} \int_{\partial B_r(x^0)}|\nabla\mathbf{u}\cdot\nu-\frac{2\mathbf{u}}{r}|^2d\mathcal{H}^{n-1}\\
			&&+ \ \ \frac{2}{r^{n+3}}\int_{B_r(x^0)}F'(|\mathbf{u}|)|\mathbf{u}|\chi_{\{|\mathbf{u}|>0\}}-F(|\mathbf{u}|)dx\\
			&\geq& 0.
		\end{eqnarray*}
	\end{lemma}
	
	\begin{remark} \
		Since for the general function $F$, this term has no Scaling properties. Therefore, based on \cite{asu} and  \cite{w99}, inspired by G. S. Weiss \cite{w98}, we obtain the Weiss-type monotonicity formula using the following Lemma regarding domain variation of the functional $E(\mathbf{u})$.
	\end{remark}

	\begin{lemma}\label{lem3.3}
		Let $\mathbf{u}$ be the minimizer of $E(\mathbf{u})$ in $B_r(x^0)$. Then it infers
		\begin{equation}\label{eq3.1}
			\int_{B_r(x^0)}|\nabla\mathbf{u}|^2 \mathrm{div}\xi -2\nabla\mathbf{u}D\xi\cdot\nabla{\mathbf{u}}   + F(|\mathbf{u}(x)|)\mathrm{div} \xi dx=0, \ \ \forall \xi\in C^{\infty}_c \left(B_r(x^0);\mathbb{R}^n\right).
		\end{equation}
	\end{lemma}

	\begin{proof}
		Following the argument from Chapter 9 of \cite{bv}, we deduce the more general case of the functional $E(\mathbf{u})$. A direct computation gives that
		\begin{eqnarray*}
			0&=&E'(\mathbf{u})=\displaystyle\lim_{t\rightarrow 0}\frac{E(\mathbf{u}(x+t\xi(x)))-E(\mathbf{u})}{t}\nonumber\\
			&=&\int_{B_r(x^0)} \left(  |\nabla\mathbf{u}|^2 \mathrm{div}\xi -2\nabla\mathbf{u}D\xi\cdot\nabla{\mathbf{u}}  \right)dx + \displaystyle\lim_{t\rightarrow 0}\frac{1}{t}\left(\int_{B_r(x^0)} \left(F(|\mathbf{u}(x+t\xi(x))|)-F(|\mathbf{u}(x)|)\right)dx\right),\nonumber
		\end{eqnarray*}where the second term is equal to
		\begin{eqnarray*}
			&&\displaystyle\lim_{t\rightarrow 0}\frac{1}{t}\left(\int_{B_r(x^0)} F(|\mathbf{u}_t(x)|)dx-\int_{B_r(x^0)} F(|\mathbf{u}(x)|)dx\right)\nonumber\\
			&=&\displaystyle\lim_{t\rightarrow 0}\frac{1}{t}\left(\int_{B_r(x^0)} F(|\mathbf{u}_t(\Psi_t^{-1}(y))|)|\mathrm{det} D\Psi_t(y)|dy-\int_{B_r(x^0)} F(|\mathbf{u}(x)|)dx\right)\nonumber\\
			&=&\int_{B_r(x^0)} F(|\mathbf{u}(x)|)\mathrm{div} \xi dx.
		\end{eqnarray*}
		Here $\mathbf{u}_t(x)=\mathbf{u}(x+t\xi(x))$, $\Psi_t:=x+t\xi$. Therefore,
		\begin{eqnarray*}
			0 &=& \int_{B_r(x^0)}|\nabla\mathbf{u}|^2 \mathrm{div}\xi -2\nabla\mathbf{u}D\xi\cdot\nabla{\mathbf{u}}   + F(|\mathbf{u}(x)|)\mathrm{div} \xi dx.
		\end{eqnarray*}
	\end{proof}
	
	{{\it Proof of  Lemma \ref{monotonicity}.} } Using a direct computation, one can see
	\begin{eqnarray}\label{eq3.2}
		\frac{dW(\mathbf{u},x^0,r)}{dr}&=&\frac{d}{dr}\left(\frac{1}{r^{n+2}}\int_{B_r(x^0)}\left(|\nabla\mathbf{u}|^2+F(|\mathbf{u}|)\right)dx-\frac{2}{r^{n+3}}\int_{\partial B_r(x^0)}|\mathbf{u}|^2 d\mathcal{H}^{n-1}\right)\nonumber\\
		&=&-(n+2)\frac{1}{r^{n+3}}\int_{B_r(x^0)} \left(|\nabla\mathbf{u}|^2+F(|\mathbf{u}|)\right)dx+\frac{1}{r^{n+2}}\int_{\partial B_r(x^0)}|\nabla\mathbf{u}|^2+F(|\mathbf{u}|)d\mathcal{H}^{n-1}\nonumber\\
		&&+\frac{2(n+3)}{r^{n+4}}\int_{\partial B_r(x^0)}|\mathbf{u}|^2 d\mathcal{H}^{n-1}-\frac{2}{r^{n+3}}\int_{\partial B_r(x^0)}\frac{n-1}{r}|\mathbf{u}|^2d\mathcal{H}^{n-1}\nonumber\\
		&&-\frac{2}{r^{n+3}}\int_{\partial B_r(x^0)}2\mathbf{u}\nabla\mathbf{u}\cdot\nu d\mathcal{H}^{n-1}.
	\end{eqnarray}
	
	Now recalling Lemma \ref{lem3.3} and taking  $$\xi_{\tau}(x):=\eta_{\tau}(|x-x^0|)(x-x^0),$$ where $\tau$ is small and $\eta_{\tau}(t)=\max\left\{ 0, \min\{1,\frac{r-t}{\tau}\}\right\}$, then it infers
	\begin{equation*}
		D\xi(x) =(\partial_j\xi_{\tau}^i)(x)= \eta_\tau(|x-x^0|)\delta_{ij}+\eta_{\tau}'(|x-x^0|)\frac{x_j-x^0_{j}}{|x-x^0|}(x_i-x^0_{i})
	\end{equation*}
	and
	\begin{equation*}
		\mathrm{div}\xi(x)=(\partial_i\xi_{\tau}^i)(x)= n\eta_\tau(|x-x^0|)+\eta_{\tau}'(|x-x^0|)(x_i-x^0_{i}),
	\end{equation*}
	where $\xi_{\tau}^i(x):=\eta_{\tau}(|x-x^0|)(x_i-x_i^0)$, $\delta_{ij}=0$ as $i\neq j$; $\delta_{ij}=1$ as $i=j$ ($i,j=1,2,...,n$) and $x=(x_1,x_2,...,x_n)$, $x^0=(x^0_1,x^0_2,...,x^0_n)$. The above two equalities together with (\ref{eq3.1}), as $\tau\rightarrow 0$, implies that
	\begin{equation}\label{eq3.3}
		\begin{aligned}
			0=&n\int_{ B_r(x^0)}|\nabla\mathbf{u}|^2+F(|\mathbf{u}|)dx - r\int_{\partial B_r(x^0)}|\nabla\mathbf{u}|^2+F(|\mathbf{u}|)d\mathcal{H}^{n-1}\\
			&-2\int_{ B_r(x^0)}|\nabla\mathbf{u}|^2dx +2r\int_{\partial B_r(x^0)}(\nabla\mathbf{u}\cdot\nu)^2 d\mathcal{H}^{n-1}.
		\end{aligned}
	\end{equation}
	Combining  \eqref{eq3.2} and \eqref{eq3.3}, one gets
	\begin{equation}\nonumber
		\begin{aligned}
			\frac{dW(\mathbf{u},x^0,r)}{dr}=&-r^{-n-3}\Big(2\int_{B_r(x^0)}|\nabla\mathbf{u}|^2dx-2r\int_{\partial B_r(x^0)}(\nabla\mathbf{u}\cdot\nu)^2 d\mathcal{H}^{n-1}+2\int_{B_r(x^0)}|\nabla\mathbf{u}|^2+F(|\mathbf{u}|) dx\Big)\\
			&-8r^{-n-4}\int_{\partial B_r(x^0)}|\mathbf{u}|^2d\mathcal{H}^{n-1}+\frac{4}{r^{n+3}}\int_{\partial B_r (x^0)} \mathbf{u}\cdot\nabla\mathbf{u}\cdot\nu d \mathcal{H}^{n-1}\\
			=&-r^{-n-3} \int_{B_r(x^0)}4|\nabla\mathbf{u}|^2+2F(|\mathbf{u}|)dx\\
			&+\frac{2}{r^{n+2}}\Big(\int_{\partial B_r(x^0)}(\nabla\mathbf{u}\cdot\nu)^2 d\mathcal{H}^{n-1} +\int_{\partial B_r(x^0)}\frac{4|\mathbf{u}|^2}{r^2}d\mathcal{H}^{n-1}-\int_{\partial B_r(x^0) } 2\frac{\mathbf{u}}{r}\cdot\nabla\mathbf{u}\cdot\nu d \mathcal{H}^{n-1} \Big).\\
		\end{aligned}
	\end{equation}
	This, together with using integration by parts for \eqref{eq1.1},  infers
	\begin{eqnarray*}
		\frac{dW(\mathbf{u},x^0,r)}{dr}&=&\frac{2}{r^{n+2}}\int_{\partial B_r(x^0)}\left(\nabla\mathbf{u}\cdot\nu-\frac{2|\mathbf{u}|}{r}\right)^2 d\mathcal{H}^{n-1} \\ &&+\frac{2}{r^{n+3}}\int_{B_r(x^0)}F'(|\mathbf{u}|)|\mathbf{u}|\chi_{\{|\mathbf{u}|>0\}}-F(|\mathbf{u}|)dx\\
		&\geq&0,
	\end{eqnarray*}
	here we have used the convexity of $F$, for which $F(|\mathbf{u}|)\leq F'(|\mathbf{u}|)|\mathbf{u}|$.
	
	\qed

	Once Lemma \ref{monotonicity} is established, one may have
	
	\begin{lemma}\label{property} Let $\mathbf{u}$ be the minimizer of \eqref{eq1.0}. Then the following conclusions hold true,

		$\mathrm{(1)}$ The function $r\longmapsto W(\mathbf{u},x^0,r) $ has a right limit $W(\mathbf{u},x^0,0+)\in [-\infty, +\infty)$  as $r\rightarrow0+$ and it has also a limit $W(\mathbf{u},x^0,+\infty)\in (-\infty, +\infty]$ as $r\rightarrow+\infty$  when $D=\mathbb{R}^n$ .
		
		$\mathrm{(2)}$ Let $0<r_k\rightarrow 0$ be a sequence such that the blow-up sequence
		$${\mathbf{u}}_k(x)=\frac{\mathbf{u}(x^0+r_kx)}{r_k^2},$$
		converges weakly in $W^{1,2}_{loc}(\mathbb{R}^n, \ \mathbb{R}^m)$ to $\mathbf{u}_0$. Then ${\mathbf{u}}_0$ is a homogeneous function of degree 2.
		Moreover
		$$W(\mathbf{u},x^0,0+)=\int_{B_1(0)} f(0)|\mathbf{u}_0|  \ dx\geq0.$$
		And $W(\mathbf{u},x^0,0+)=0$ implies that $\mathbf{u}\equiv0$ in $B_{\delta}(x^0)$ for some $\delta>0$.
		
		$\mathrm{(3)}$ The function $x\longmapsto W(\mathbf{u},x,0+)$ is upper-semicontinuous.
		
	\end{lemma}
	
	\begin{proof}(1)It follows directly from the monotonicity of the Weiss's energy functional (Lemma 3.1).
		
		(2) Noticing the assumption of $\{\mathbf{u}_k\}$,  it infers that $\{\mathbf{u}_k\}$ is bounded in $W^{1,2}_{loc}(\mathbb{R}^n,\mathbb{R}^m)$, and the limit $W(\mathbf{u},x^0,0+)$ is finite. Applying Lemma 3.1, for any $0<\rho<\sigma<+\infty$ we get that
		\begin{eqnarray*}
			0&\leq&\int_{\rho}^{\sigma}\frac{1}{r^{n+2}}\int_{\partial B_r(0)}(\nabla{\mathbf{u}}_k\cdot \nu-\frac{2{\mathbf{u}}_k}{r})^2d\mathcal{H}^{n-1}dr\\
			&=&\int_{\rho}^{\sigma}\frac{1}{r^{n+2}}\int_{\partial B_r(0)}\big|\frac{\nabla \mathbf{u}(x^0+r_kx)}{r_k}\cdot\nu-2\frac{{\mathbf{u}}(x^0+r_kx)}{r_k^2r}\big|^2d_x\mathcal{H}^{n-1}dr\\
			&=&\int_{\rho}^{\sigma}\frac{1}{r^{n+4}}\int_{\partial B_{rr_k}(x^0)}\frac{1}{r_k^{n-1}}\frac{1}{r_k^4}\mid y\cdot\nabla {{\mathbf{u}}}-2{{\mathbf{u}}}\mid^2 d_y\mathcal{H}^{n-1}dr\\
			&\leq& \int_{\rho r_k}^{\sigma r_k}\frac{r_k}{\tau^{n+4}}\int_{\partial B_{\tau}(x^0)}\mid y\cdot\nabla {{\mathbf{u}}}-2{{\mathbf{u}}}\mid^2d_y\mathcal{H}^{n-1}d\tau\\
			&+&\frac{r_k}{\tau^{n+3}}\int_{B_{\tau}(x^0)}F'(|{{\mathbf{u}}}(y)|)|{{\mathbf{u}}}(y)|\chi_{\{|{{\mathbf{u}}}(y)|>0\}}-F(|{{\mathbf{u}}}(y)|)dy\\
			&=& r_k \big(W({{\mathbf{u}}},x^0,\sigma r_k)-(W({{\mathbf{u}}},x^0,\rho r_k)\big)\rightarrow 0, \quad \text{as} \quad k\rightarrow \infty,
		\end{eqnarray*}which shows
		$$\int_{\rho}^{\sigma}\frac{1}{r^{n+2}}\int_{\partial B_r(0)}\big|\nabla{{\mathbf{u}}}_k\cdot \nu-\frac{2{{\mathbf{u}}}_k}{r}\big|^2d\mathcal{H}^{n-1} dr \rightarrow 0,\quad \text{as} \quad k\rightarrow \infty.$$ Hence ${u}_0$ is a homogeneous function of degree 2.
		Moreover, the homogeneity of ${{\mathbf{u}}}_0$ gives that
		\begin{align}\label{eq3.4}
			\displaystyle\lim_{r\rightarrow 0+}W(\mathbf{u},x^0,r)&=\displaystyle\lim_{r\rightarrow 0+}\frac{1}{r^{n+2}}\int_{B_r(x^0)}|\nabla\mathbf{u}|^2+F(|\mathbf{u}|) dx -\frac{2}{r^{n+3}}\int_{\partial B_r(x^0)}|\mathbf{u}|^2 d\mathcal{H}^{n-1}\nonumber\\
			\nonumber\\
			&=\displaystyle\lim_{r\rightarrow 0+}\int_{B_1(0)}|\nabla\mathbf{u}_r(x)|^2+\frac{1}{r^2}F(r^2|\mathbf{u}_r(x)|)dx-2\int_{\partial B_1(0)}|\mathbf{u}_r(x)|^2d\mathcal{H}^{n-1}.  	
	\end{align}
		Furthermore, we have
		$$ \displaystyle\lim_{r\rightarrow 0+}\int_{B_1(0)}\frac{1}{r^2}F(r^2|\mathbf{u}_r|)dx \leq \displaystyle\lim_{r\rightarrow 0+}\int_{B_1(0)} \frac{1}{r^2}F'(r^2|\mathbf{u}_r|)\cdot r^2|\mathbf{u}_r|dx=\int_{B_1(0)}F'(0)|\mathbf{u}_0(x)|dx,$$
		where the convexity of $F$ has been applied. On the other hand, the convexity of $F$ (see Remark \ref{rem1.1}) also gives that
		\begin{eqnarray*}
			\displaystyle\lim_{r\rightarrow 0+}\int_{B_1(0)}\frac{1}{r^2}F(r^2|\mathbf{u}_r|)dx\geq\displaystyle\lim_{r\rightarrow 0}\int_{B_1(0)}F'(0)|\mathbf{u}_r|dx=\int_{B_1(0)}F'(0)|\mathbf{u}_0(x)|dx.
		\end{eqnarray*}
		Therefore, there holds
		$$ \displaystyle\lim_{r\rightarrow 0+}\int_{B_1(0)}\frac{1}{r^2}F(r^2|\mathbf{u}_r|)dx=\int_{B_1(0)}F'(0)|\mathbf{u}_0(x)|dx, $$ which together with \eqref{eq3.4}, shows
		\begin{equation*}
			\begin{aligned}
				&\displaystyle\lim_{r\rightarrow 0+}W(\mathbf{u},x^0,r)\\
				=&\int_{B_1(0)}|\nabla\mathbf{u}_0|^2+F'(0)|\mathbf{u}_0|dx-2\int_{\partial B_1(0)}|\mathbf{u}_0|^2d\mathcal{H}^{n-1}\\
				=&-\int_{B_1(0)}\Delta{\mathbf{u}_0}\cdot {\mathbf{u}_0}dx+\int_{\partial B_1(0)}\frac{\partial {\mathbf{u}_0}}{\partial\nu}\cdot{\mathbf{u}_0}d\mathcal{H}^{n-1}+\int_{B_1(0)}F'(0)|\mathbf{u}_0(x)|dx-2\int_{\partial B_1(0)}|\mathbf{u}_0|^2d\mathcal{H}^{n-1}\\
				=&\int_{B_1(0)}\left(F'(0)-\Delta{\mathbf{u}_0}\cdot\frac{{\mathbf{u}_0}}{|{\mathbf{u}_0}|}\right)|{\mathbf{u}_0}| dx +\int_{\partial B_1(0)}\left(\frac{\partial\mathbf{u}_0}{\partial\nu}-2\mathbf{u}_0\cdot {\mathbf{u}_0}\right)d\mathcal{H}^{n-1}\\
				=&\int_{B_1(0)}\frac{F'(0)}{2}|\mathbf{u}_0|dx\geq 0.
			\end{aligned}
		\end{equation*}
		Now, notice that the equation \eqref{eq1.1}, one can obtain
		\begin{equation*}
			\Delta \mathbf{u}_r=\frac{F'(r^2|\mathbf{u}_r|)}{2}\frac{\mathbf{u}_r}{|\mathbf{u}_r|}\chi_{\{|\mathbf{u}_r|>0\}},
		\end{equation*}
		as $r\to 0+$, which implies
		\begin{equation}\label{eq3.7}
			\Delta \mathbf{u}_0=\frac{F'(0)}{2}\frac{\mathbf{u}_0}{|\mathbf{u}_0|}\chi_{\{|\mathbf{u}_0|>0\}}= f(0)\frac{\mathbf{u}_0}{|\mathbf{u}_0|}\chi_{\{|\mathbf{u}_0|>0\}}.
		\end{equation}
		Therefore, when $W(\mathbf{u},x^0,0+)=0$, it is easy to obtain $\mathbf{u}\equiv 0$ in a small ball $B_{\delta}(x^0)$.
		
		(3) With the definition of upper-semicontinuity, utilizing the monotonicity property and the absolute continuity of Lebesgue integration, we can obtain upper-semicontinuity, for more details we refer to the proof of \cite[Lemma 2]{asu}.

	\end{proof}
	
	To the end of this section, we will give a quadratic growth estimate for $\mathbf{u}$, using the Weiss-type monotonicity formula.
	\begin{proposition}\label{growth}
		Any solution $\mathbf{u}$ to the system \eqref{eq1.1} in $B_1(0)$ satisfies
		\begin{eqnarray*}
			|\mathbf{u}(x)|\leq C dist^2 (x,\Gamma_0(\mathbf{u})) \quad \text{and} \quad |\nabla \mathbf{u}(x)|\leq C dist (x,\Gamma_0(\mathbf{u})) \quad \text{for every} \quad x\in B_{\frac{1}{2}}(0),
		\end{eqnarray*}
		where the constant $C>0$ depends only on $n$, $c_0$ , $C_0$  ( $c_0$ and $C_0$ are  given in \eqref{eq1.3}) and $\displaystyle E(\mathbf{u},0,1).$
		
	\end{proposition}
	
	\begin{proof}
		Equivalently it suffices to verify that there exists  a constant $C_1=C_1(n, c_0,C_0,E(\mathbf{u},0,1))>0$ with $c_0$  given in \eqref{eq1.3} such that
		\begin{equation}\label{eq3.7-1}\displaystyle\sup_{B_r(x^0)}|\mathbf{u}|\leq C_1 r^2\quad \text{and}\quad \displaystyle\sup_{B_r(x^0)}|\nabla\mathbf{u}|\leq C_1 r
		\end{equation}
		for every $x^0\in \Gamma_0(\mathbf{u})\cap B_{1/2}(0)$ and every $r\in (0,1/4]$.
		
		In the following one firstly gives the claim that there exists a $ C_2=C_2(n, c_0,C_0,E(\mathbf{u},0,1))>0$  such that
		\begin{equation}\label{eq3.8}
			\frac{1}{r^n}\int_{B_r(x^0)}|\mathbf{u}|dx \leq C_2 r^2
		\end{equation}
		for every $x^0\in \Gamma_0(\mathbf{u})\cap B_{1/2}(0)$ and every $r\in (0,1/2]$.
		
		This claim can imply \eqref{eq3.7-1} due to the Proposition \ref{bounded}. Now we turn to the proof of \eqref{eq3.8}. Firstly using the monotonicity formula Lemma \ref{monotonicity} and \eqref{eq1.3}, then  for any $ x^0\in B_{\frac{1}{2}}(0)\cap \Gamma_0(\mathbf{u})$ and $ r\leq{\frac{1}{2}},$ we obtain
		\begin{eqnarray*}
			W(\mathbf{u},x^0,r)&\leq& W(\mathbf{u},x^0,\frac{1}{2})\\
			&\leq&2^{n+2}\int_{B_{\frac{1}{2}}(x^0)}|\nabla \mathbf{u}|^2+F(|\mathbf{u}|)dx\\
			&\leq&2^{n+2}\int_{B_1(0)}|\nabla \mathbf{u}|^2+F(|\mathbf{u}|)dx\\
			&=& 2^{n+2}E(\mathbf{u},0,1).
		\end{eqnarray*}
		Then, due to the convexity of $F$, one can see
		\begin{equation*}
			\begin{aligned}
				\frac{c_0}{r^{n+2}}\int_{B_r(x^0)}|\mathbf{u}| dx \leq&\frac{1}{r^{n+2}} \int_{B_r(x^0)} F(|\mathbf{u}|) dx \\
				=&W(\mathbf{u},x^0,r)-\frac{1}{r^{n+2}} \int_{B_r(x^0)} |\nabla\mathbf{u}|^2 dx+\frac{2}{r^{n+3}} \int_{\partial B_r(x^0)}|\mathbf{u}|^2d\mathcal{H}^{n-1}\\
				\leq&2^{n+2 }{E(\mathbf{u},0,1)}+\frac{2}{r^{n+3}} \int_{\partial B_r(x^0)}|\mathbf{u}-S_{x^0}\mathbf{p})|^2 d\mathcal{H}^{n-1},
			\end{aligned}
		\end{equation*}
		for each $\mathbf{p}\in \mathcal{H}:=\{\mathbf{p}=(p_1,p_2,...,p_m)|  p_j \ \text{is a 2-homogeneous harmonic polynomial}, \forall 1\leq j\leq m\}$, $S_{x^0}\mathbf{P}(x):=\mathbf{P}(x-x^0)$.
		Next, we only need to prove the boundness of the terms on the right side of above inequality.
		Note that
		\begin{equation*}
			\begin{aligned}
				I(\mathbf{p},x_0,r):=&\frac{1}{r^{n+3}} \int_{\partial B_r(x^0)}|\mathbf{u}-S_{x^0}\mathbf{p})|^2 d\mathcal{H}^{n-1}\\
				=&\int_{\partial B_1(0)}\left|\displaystyle\frac{\mathbf{u}(x^0+rx)}{r^2}-\mathbf{p}(x)\right|^2 d\mathcal{H}^{n-1}.
			\end{aligned}
		\end{equation*}
		Let $\mathbf{p}_{x^0}$ be the minimizer of $I(\mathbf{p},x_0,r)$ in $\mathcal{H}$.
		On the one hand, suppose towards a contradiction that there is a sequence of solutions $\mathbf{u}_k$ and a sequence of points $x^k\in B_{\frac{1}{2}(0)}\cap \Gamma_0(\mathbf{u}_k)$ as well as $r_k\to 0$ such that $E(\mathbf{u},0,1)$ are uniformly bounded,
		$$M_k:=\int_{\partial B_1(0)}\left|\displaystyle\frac{\mathbf{u}_k(x^k+r_kx)}{r_k^2}-\mathbf{p}_{x^k}(x)\right|^2 d\mathcal{H}^{n-1}\to \infty.$$
		Let $$\mathbf{w}_k:=\Big(\displaystyle\frac{\mathbf{u}_k(x^k+r_kx)}{r_k^2}-\mathbf{p}_{x^k}(x)\Big)/M_k,$$
		it leads to $\|\mathbf{w}_k\|_{L^2(\partial B_1(0);\mathbb{R}^m)}=1$. On the other hand, from Rellich's theorem it follows that $\mathbf{w}_k$ converges strongly in $L^2(\partial B_1(0);\mathbb{R}^m)$, and due to $\mathbf{p}_{x^k}$ be the minimizer of $I(\mathbf{p},x_0,r)$ in $\mathcal{H}$ we can obtain $\int_{\partial B_1(0)}|\mathbf{w}_0|^2 d\mathcal{H}^{n-1}=0$, therefore it's contradictory. We will omit specific details of the proof by contradiction, the readers can refer to the proof of \cite[Theorem 2]{asu}.
	\end{proof}

	\vspace{25pt}
	
	\section{Homogeneous  global solution}
	In this section, the main purpose is to consider the characterization of the global solutions of system \eqref{eq1.1} and establish the relationship between the set of regular free boundary points $\mathcal{R}_{\mathbf{u}}$ and $\Gamma_0(\mathbf{u})$.
	
	Firstly, using the second part of Lemma \ref{property}, we have obtained that the blow-up limit $\mathbf{u}_0$ is homogeneous function of degree 2. Hence, one defines homogeneous global solutions of degree 2 as follows.
	\begin{definition}\label{global}
		$\mathbf{v}\in W^{1,2}(\mathbb{R}^n,\mathbb{R}^m)$ is called \textit{a homogeneous global solution of degree 2}, provided that it satisfies the following two conditions.
		
		(1) ({\it Homogeneity})
		$$ \mathbf{v}(\lambda x)=\lambda^2\mathbf{v}(x) \ \ \  \text{for all } \ \ \lambda>0 \ \ \text{and} \ \  x\in \mathbb{R}^n.$$
		
		(2) It is a weak solution to
		\begin{equation}\label{eq4.1}
			\Delta \mathbf{v}= f(0)\frac{\mathbf{v}}{|\mathbf{v}|}\chi_{\{|\mathbf{v}|>0\}}\quad \text{in}\quad \mathbb{R}^n.
		\end{equation}
	\end{definition}
	
	The following result gives that the solution has to be a half space solution, provided that its support lies above the plane $x_n=-\delta$ for some $\delta>0$ in $B_1(0)$.
	\begin{proposition}\label{prop4.1}
		Assume that $\mathbf{u} $ is a homogeneous global solution of degree 2. Then there exists a positive constant $\delta=\delta(n, \sup_{B_1(0)}{|\mathbf{u}|} )$ such that if $\left(B_1\cap \text{supp}\   \mathbf{u} \right) \subset \{x_n>-\delta\}$, then $\mathbf{u}\in\mathbb{H}$, i.e. $\mathbf{u}=\displaystyle\frac{f(0)\max (x\cdot\nu,0)^2}{2}\mathbf{e}^1$ for some unit vector $\nu\in \mathbb{R}^n$ and $\mathbf{e}\in \mathbb{R}^m$ (See Fig.2).
	\end{proposition}
	
	\begin{figure}[!h]
		\includegraphics[width=100mm]{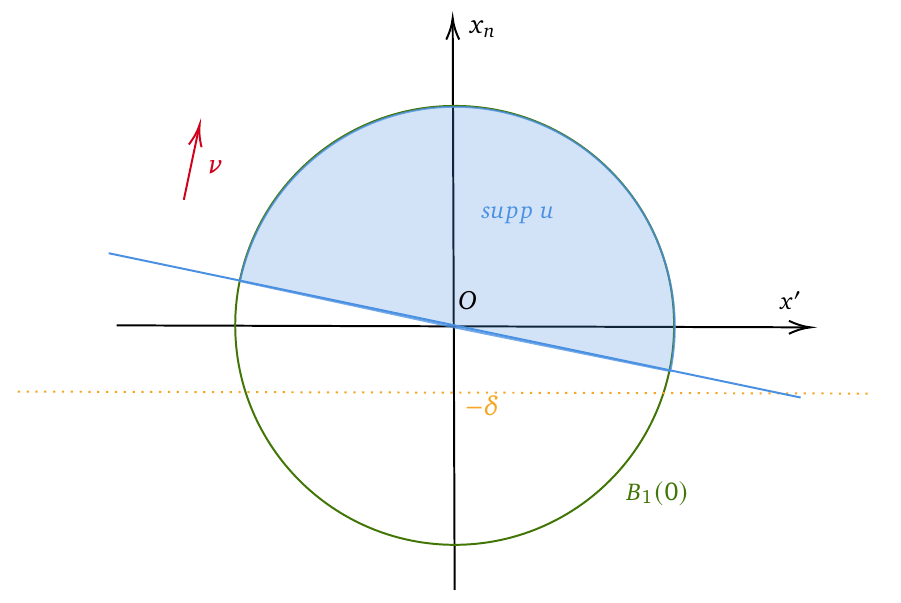}
		\caption{The support of $\mathbf{u} $}
	\end{figure}
	
	\begin{proof}
		The proof  is quite similar to the Proposition 3 in\cite{asu}. Firstly, we consider that each component $u^i$ of $\mathbf{u}$ is a solution of
		$$L_i(u^i):=-\Delta' u^i + f(0)\frac{u^i}{|\mathbf{u}|}=2nu^i$$
		in every connected component $\Omega'\subset \partial B_1\cap \{|\mathbf{u}|>0\}$, where $\Delta'$ is Laplace-Beltrami operator on the unit sphere in $\mathbb{R}^n$, which is inspired by the approach of \cite[Proposition 3]{asu}. Then, the Appendix gives that for each connected component $\Omega$ of $B_1\cap \{|\mathbf{u}|>0\}$ there exists a unit vector $\mathbf{a}=(a_1,a_2,...,a_m)$ such that $\mathbf{u}= \mathbf{a}|\mathbf{u}(x)|$ and $\Delta|\mathbf{u}|=f(0)$ in $\Omega$. As the proof similar to the \cite[Proposition 3]{asu}, we consider the two cases, on the one hand $|\nabla \mathbf{u}|=0$ on the $\partial \Omega$. On the other hand, $|\nabla \mathbf{u}|\neq 0$ on the $\partial \Omega$. The two cases yield that $\mathbf{u}\in \mathbb{H}$. The only difference between our case and \cite{asu} is the constant $f(0)$. Therefore, we omit the details here and interested readers can be referred to \cite{asu}.
	\end{proof}
	
	\begin{remark} Proposition \ref{prop4.1}, as the inverse of Proposition \ref{pro2.3}, tells us that if the support of a 2-homogeneous solution $\mathbf{u}$ is a small perturbation of the support of a half space solution $\mathbf{h}$, then this 2-homogeneous solution $\mathbf{u}$ is a half space solution $\mathbf{h}$.
	\end{remark}

	A direct corollary from Proposition \ref{pro2.3} and Proposition \ref{prop4.1} is as follows.
	\begin{corollary}\label{isolated} The half space solutions are (in the $L^1(B_1(0);\mathbb{R})^m$-topology) isolated with the class of homogeneous global solutions of degree 2.
	\end{corollary}
	\begin{proof}
		Let $\mathbf{u}$ be a homogeneous global solution of degree 2 such that exists an $\epsilon>0$ small enough such that $\|\mathbf{u}-\mathbf{h}\|_{L^1(B_1;\mathbb{R}^m)}\leq \epsilon$ for some $\mathbf{h}\in \mathbb{H}$.	
		Without loss of generality, one may assume that $$\mathbf{h}(x)=\frac{f(0)\text{max}(x_n,0)^2}{2}\mathbf{e}^1.$$ Hence from  Proposition \ref{pro2.3} and Proposition 4.2, it yields that $\mathbf{u}\in \mathbb{H}$ for sufficiently small $\epsilon$.
	\end{proof}

	The following two propositions give important estimates of balanced energy for homogeneous global solutions of degree 2.
	
	\begin{proposition}\label{prop4.4}
		Let $\alpha_n$ be the constant defined in Remark \ref{rem1.7}, then
		\begin{equation}\label{eq4.2}
			\alpha_n=\frac{{f^2(0)} \mathcal{H}^{n-1}(\partial B_1)}{2n(n+2)}= \frac{{f^2(0)}|B_1|}{2(n+2)}.
		\end{equation}
		Let $\mathbf{u}$ be a homogeneous global solution of degree 2. Then one has
		\begin{equation}\label{eq4.3}
			M(\mathbf{u})\geq  \alpha_n \frac{|B_1\cap \{|\mathbf{u}|>0\}|}{|B_1|}.
		\end{equation}
		In particular,
		\begin{equation}\label{eq4.4}
			M(\mathbf{u})\geq {\alpha_n}, \ \  \text{if}\  |\mathbf{u}|>0 \ \ \text{a.e. \ in }\  B_1.
		\end{equation}
		
	\end{proposition}
	
	\begin{proof}
		Firstly, to show \eqref{eq4.2},  we choose a unit vector $\mathbf{e}\in \partial B_1 \subset \mathbb{R}^m$ and let $\mathbf{h}(x)=\mathbf{e} f(0)\max (x_n,0)^2/2$. Obviously, by the definition of $\alpha_n$ in Remark \ref{rem1.7},  it infers
		\begin{equation*}
			\frac{\alpha_n}{2}=M(\mathbf{h})=\int_{B_{1}}|\nabla \mathbf{h}|^2 +2f(0)|\mathbf{h}|dx-2\int_{\partial B_1} |\mathbf{h}|^2 d \mathcal{H}^{n-1}.
		\end{equation*}
		Obtain the following equation through divergence theorem since $\mathbf{h}$ is homogeneous function of degree 2,
		\begin{equation*}
			M(\mathbf{h})=\int_{B_{1}} (2f(0)-\Delta \mathbf{h}\displaystyle\frac{\mathbf{h}}{|\mathbf{h}|}) |\mathbf{h}| dx.
		\end{equation*}
		And then combined with the definition of $\mathbf{h}$, it yields that
		\begin{equation*}
			M(\mathbf{h})=\frac{f^2(0)}{2}\int_{B_1} \max (x_n,0)^2 dx=\frac{f^2(0)}{4}\int_{B_1} x_n^2 dx=\frac{f^2(0)}{4n}\int_{B_1} |x|^2 dx =\frac{f^2(0)  |B_1|}{4(n+2)},
		\end{equation*}
		In fact the symmetry of $B_1$ gives that $n\int_{B_1} x_n^2dx= \int_{B_1}|x|^2dx$ and
		$$\int_{B_{1}}|x|^2=\int_0^1\int_{\partial B_r} r^2 d\mathcal{H}^{n-1}dr=\frac{1}{n+2} \mathcal{H}^{n-1}(\partial B_1)=\frac{n}{n+2}|B_1|.$$
		Therefore which shows \eqref{eq4.2}.
		
		Next, the proof of inequality \eqref{eq4.3} can be obtained as in \cite[Proposition 4]{asu}, from using the homogeneity of $\mathbf{u}$ and the divergence theorem. In particular, the estimate \eqref{eq4.4} holds true.
	\end{proof}

	\begin{corollary}\label{coro4.5}
		Let $\mathbf{u}$ be a homogeneous global solution of degree 2 as in Definition \ref{global}. Then,
		\begin{equation}\label{eq4.7}
			M(\mathbf{u})\geq \alpha_n \max \left(\frac{1}{2}, \frac{| B_1 \cap \{|\mathbf{u}|>0\}|}{ |B_1|}\right)\geq  \frac{\alpha_n}{2},
		\end{equation}
		and $\displaystyle M(\mathbf{u})=\frac{\alpha_n}{2}$ implies that $\mathbf{u}\in \mathbb{H}.$ Moreover, $$ \frac{\alpha_n}{2}< \bar{\alpha}_n :=\inf \{ M(\mathbf{v}): \mathbf{v}\  \text{ is a homogeneous solution of degree 2, but} \ \mathbf{v}\notin \mathbb{H}\}.$$
	\end{corollary}
	\begin{proof}
		If $ |B_1 \cap \{|\mathbf{u}|=0\}|=0$, then \eqref{eq4.4} implies the conclusion \eqref{eq4.7} immediately.
		
		On the other hand, if $ |B_1 \cap \{|\mathbf{u}|=0\}|\neq0$, then by Proposition \ref{non-degeneracy} (the non-degeneracy property), $\{|\mathbf{u}|=0\}$ must contain some open ball $B_{\rho} (y)$ and we may choose  a point $z\in \partial B_{\rho}(y)\cap \partial \{|\mathbf{u}|>0\}$.
		
		Let $\mathbf{u}_0$ be a blow-up limit of $\mathbf{u}$ at $z$. Obviously, supp $\mathbf{u}_0$ is contained in a half-space, then it follows from Proposition \ref{prop4.1} that $\mathbf{u}_0\in \mathbb{H}$.
		By observing the form of balanced energy, we can set
		$$\Phi (\mathbf{u},x,r)=\frac{1}{r^{n+2}}\int_{B_r(x)}|\nabla\mathbf{u}|^2+F'(0)|\mathbf{u}| dx-\frac{2}{r^{n+3}}\int_{\partial B_r(x)}|\mathbf{u}|^2 d\mathcal{H}^{n-1}.$$
		
		We can deduce that the limit $\Phi(\mathbf{u},x^0,+\infty)$ does not depend on the choice of $x^0$, since $\mathbf{u}_0$ is a homogeneous global solution of degree 2, $|\mathbf{u}(x)|\leq C|x|^2$ and $|\nabla\mathbf{u}(x)|\leq C|x|$ for every $x\in \mathbb{R}^n$.

		Recalling Remark \ref{rem1.7}, Lemma \ref{monotonicity} and the homogeneity of $\mathbf{u}$,  one can obtain
		\begin{eqnarray*}
			\frac{\alpha_n}{2}=M(\mathbf{u}_0)&=&\int_{B_1(0)}|\nabla\mathbf{u}_0|^2+F'(0)|\mathbf{u}_0|\ dx-2\int_{\partial B_1(0)}|\mathbf{u}_0|^2 d\mathcal{H}^{n-1}\\
			&=&\displaystyle\lim_{r\rightarrow 0}\frac{1}{r^{n+2}}\int_{B_r(z)}|\nabla\mathbf{u}|^2+F'(0)|\mathbf{u}|\ dx-\frac{2}{r^{n+3}}\int_{\partial B_r(z)}|\mathbf{u}|^2 d\mathcal{H}^{n-1}\\
			&=& \Phi(\mathbf{u},z,0^+)\leq\Phi(\mathbf{u},z,+\infty)=\Phi(\mathbf{u},0,+\infty)=M(\mathbf{u}),
		\end{eqnarray*}
		which finishes the proof of \eqref{eq4.7}.
		Indeed, if $F(|\mathbf{u}|)$ of $W(\mathbf{u},x,r)$ equals to $F'(0)|\mathbf{u}|$, then $\Phi (\mathbf{u},x,r)=W(\mathbf{u},x,r)$, hence by Lemma \ref{monotonicity} we have that
		$$\frac{d\Phi(\mathbf{u},x^0,r)}{dr} =\frac{2}{r^{n+2}} \int_{\partial B_r(x_0)}\left(\nabla\mathbf{u}\cdot\nu-\frac{2\mathbf{u}}{r}\right)^2d\mathcal{H}^{n-1}\geq 0,$$
		thus, it follows that
		\begin{eqnarray}\label{eq4.12}
			\Phi(\mathbf{u},z,0^+)\leq\Phi(\mathbf{u},x^0,r)\leq \Phi(\mathbf{u},0,+\infty), \quad \text{for}\quad z\in \partial B_\rho(y)\cap\partial\{|\mathbf{u}|>0\}.
		\end{eqnarray}

		In the following, we shall prove that  $\displaystyle M(\mathbf{u})=\displaystyle\frac{\alpha_n}{2}$ implies $\mathbf{u}\in \mathbb{H}$.  In fact,
		let $y = z +\rho \mathbf{e}$ with a unit vector $\mathbf{e}$. It follows from the homogeneity of $\mathbf{u}$ that $\mathbf{e}$ is orthogonal to $y-z$. There are two cases to be considered.
		
		Case a) $z = 0$. Then  $\mathbf{u}(z)=\mathbf{u}(0)=0$. Since $M(\mathbf{u})=\displaystyle\frac{\alpha_n}{2}$, the estimate gives \eqref{eq4.7} and then it infers
		$$\frac{|B_1 \cap \{|\mathbf{u}|>0\}|}{|B_1|}\leq \frac{1}{2},$$
		which implies
		$$\frac{ |B_1 \cap \{|\mathbf{u}|=0\}|}{|B_1|}\geq \frac{1}{2}.$$
		Hence if $z = 0$, $|\mathbf{u}(x)|=0$ in a half space $(x\cdot\mathbf{e})>0,$ which concludes that $\mathbf{u}\in \mathbb{H}$ from Proposition \ref{prop4.1}.
		
		Case b)  $|z| > 0$. Owing to \ref{eq4.12} and $\displaystyle M(\mathbf{u})=\displaystyle\frac{\alpha_n}{2}$, it follows that
		$$\frac{\alpha_n}{2}=M(\mathbf{u}_0)= \Phi(\mathbf{u},z,0^+)\leq\Phi(\mathbf{u},z,+\infty)=\Phi(\mathbf{u},0,+\infty)=M(\mathbf{u})=\frac{\alpha_n}{2}.$$
		Hence $\Phi(\mathbf{u},z,r)$ does not depend on $r$. Next we take the different $r$, and take $r=k\in \mathbb{R}$ yields the following equation
		\begin{eqnarray*}
			\Phi (\mathbf{u},x,k)&=&\frac{1}{r^{n+2}}\int_{B_r(x)}|\nabla\mathbf{u}|^2+F'(0)|\mathbf{u}| \ dx-\frac{2}{r^{n+3}}\int_{\partial B_r(x)}|\mathbf{u}|^2 d\mathcal{H}^{n-1}\\
			&=&\int_{B_1(0)}\left|\nabla\frac{\mathbf{u}(z+kx)}{k^2}\right|^2+F'(0)\left|\frac{\mathbf{u}(z+kx)}{k^2}\right| \ dx -\int_{\partial B_1(0)}\left|\frac{\mathbf{u}(z+kx)}{k^2}\right|^2 d\mathcal{H}^{n-1}.\\
		\end{eqnarray*}
		On the other hand, taking $r=1$ gives to another equation
		\begin{equation*}
			\Phi (\mathbf{u},x,1)=\int_{B_1(0)}|\nabla\mathbf{u}(z+x)|^2+F'(0)|\mathbf{u}(z+x)| \ dx-\int_{\partial B_1(0)}|\mathbf{u}(z+x)|^2 d\mathcal{H}^{n-1}.
		\end{equation*}
		This concludes that $\mathbf{u}(z + kx) =k^2\mathbf{u}(z+x)$ for any $k>0,x\in \mathbb{R}^n$. Noting that the homogeneity of $\mathbf{u}$ also gives that $\mathbf{u}(z + kx) =k^2\mathbf{u}(z/k+x)$, then it yields that
		$$\mathbf{u}(z+x)=\mathbf{u}(z/k+x) \ \text{for any} \ k>0,x\in \mathbb{R}^n,$$  which implies that $\mathbf{u}$ is constant in direction of vector $z$. Since ${e}$ is orthogonal to $y-z$,  then $\mathbf{u}\in \mathbb{H}$.

	\end{proof}

	As a direct  corollary of the upper-semicontinuity of solution, we can easily deduce the following results.
	\begin{corollary}\label{coro5.8}
		The set of regular free boundary points $\mathcal{R}_{\mathbf{u}}$ is open relative to $\Gamma_0(\mathbf{u})$.
	\end{corollary}
	\begin{proof}
		For any $x_0\in \mathcal{R}_{\mathbf{u}}$, there is a $\delta=\delta(x_0)>0$, such that for any $x\in B_{\delta(x_0)}\cap \Gamma_0(\mathbf{u})$,
		$$\displaystyle\limsup_{x\rightarrow x_0} W(\mathbf{u},x,0+) \leq W(\mathbf{u},x_0,0+)=\frac{\alpha_n}{2}.$$
		Here we have used the upper-semicontinuity of Lemma \ref{property}. Recalling Corollary \ref{coro4.5} again, one has
		$$\displaystyle\lim_{r\rightarrow 0} W(\mathbf{u},x,r) \geq \frac{\alpha_n}{2},$$
		which implies $\displaystyle\lim_{r\rightarrow 0} W(\mathbf{u},x,r)=\frac{\alpha_n}{2}$ and it leads to $x\in\mathcal{R}_{\mathbf{u}}.$ Consequently, the regular set $\mathcal{R}_{\mathbf{u}}$ is open relative to $\Gamma_0(\mathbf{u})$.
	\end{proof}
	\vspace{10pt}

	\vspace{10pt}
	\section{Epiperimetric inequality}
This section is devoted to establishing {\it epiperimetric inequality}, an important role in indicating an energy decay estimate, and the uniqueness of blow-up limit in analyzing the regularity of free boundary. However, the epiperimetric inequality presented in \cite[Theorem 1]{asu} cannot be directly applied to our general non-degenerate case, as the function $F$ lacks scaling properties. Motivated by the linear case \cite{asu}, when certain assumptions on $F$ are satisfied, we carefully analyze Weiss' boundary-adjusted energy and the scaling parameter. Through this analysis, we successfully establish a new epiperimetric inequality for this situation.

	\begin{lemma}{ (Epiperimetric inequality)}\label{epiperimetric}
		There exist $ \kappa \in (0,1)$ and $ \delta>0 $, such that if $\mathbf{c}$ is  a homogeneous global function of degree 2 satisfying $$\|\mathbf{c}-\mathbf{h}\|_{W^{1,2}(B_1;\mathbb{R}^m)}+\|\mathbf{c}-\mathbf{h}\|_{L^{\infty}(B_1;\mathbb{R}^m)}\leq \delta,$$ for some $\mathbf{h}\in\mathbb{H}$, then there exists a vector $\mathbf{v}\in W^{1,2}(B_1;\mathbb{R}^m)$ with $\mathbf{v}=\mathbf{c}$ on $\partial B_1$ satisfying
		\begin{equation}\label{6.1}
			H(\mathbf{v},s)\leq (1-\kappa) H(\mathbf{c},s)+\kappa M(\mathbf{h}) \  \text{for any } \  s\in(0,  \delta^2].
		\end{equation}
	\end{lemma}
	\begin{remark}
The epiperimetric inequality can be described as an energy contraction, i.e.,
\begin{equation*}
	H(\mathbf{v}, s) - H(\mathbf{h},0^+) \leq (1 - \kappa) \left(H(\mathbf{c}, s) - H(\mathbf{h},0^+)\right) \quad \text{for any} \quad s \in (0, \delta^2].
\end{equation*}
	\end{remark}
	\begin{remark}
		Then we take special function $\mathbf{c}$ such that $\mathbf{c}=\mathbf{u}$ on $\partial B_1$, then due to $\mathbf{v}=\mathbf{c}$ on $\partial B_1$ and $\mathbf{u}$ is a minimizer of energy, we can obtain that
		 $$H(\mathbf{v},s)\geq H(\mathbf{u},s).$$
		\end{remark}
		
		\begin{remark}
		The epiperimetric inequality is crucial for obtaining the energy decay. Through it, we derive the linear energy $\left(G(r):=W(\mathbf{u},x^0,r)-W(\mathbf{u},x^0,0+)\right)$ inequality,
			$$\frac{dG(r)}{dr}\geq\frac{C}{r}G(r).$$
		Subsequently, the energy decay easily follows.
		\end{remark}

	\begin{proof}
		We will provide the proof by contradiction. Assume that for any $ \kappa\in(0,1)$ and $ \delta>0,$ there exist $\mathbf{c}$ , a homogeneous function of degree 2 and some $\mathbf{h}\in\mathbb{H}$ satisfying $$\|\mathbf{c}-\mathbf{h}\|_{W^{1,2}(B_1;\mathbb{R}^m)}+ \|\mathbf{c}-\mathbf{h}\|_{L^{\infty}(B_1;\mathbb{R}^m)}\leq \delta,$$ such  that there is $s\in (0, \delta^2]$  satisfying
		\begin{equation}
			H(\mathbf{v},s)> (1-\kappa) H(\mathbf{c},s)+\kappa M(\mathbf{h}) \quad \text{for any}\ \mathbf{v}\in \mathbf{c}+W^{1,2}_0(B_1;\mathbb{R}^m).\nonumber
		\end{equation}
		Without loss of generality, one can assume that there exist sequences $\kappa_k\rightarrow 0$, $\delta_k\rightarrow 0$, $\mathbf{c}_k\in W^{1,2}(B_1;\mathbb{R}^m)$ and $\mathbf{h}_k\in \mathbb{H}$ such that $\mathbf{c}_k$ is a homogeneous global function of degree 2 satisfying
		$$\|\mathbf{c}_k-\mathbf{h}_k\|_{W^{1,2}(B_1;\mathbb{R}^m)}=\inf_{h\in\mathbf{H}}\|\mathbf{c}_k-\mathbf{h}\|_{W^{1,2}(B_1;\mathbb{R}^m)}=\delta_k; \quad \|\mathbf{c}_k-\mathbf{h}_k\|_{L^{\infty}(B_1;\mathbb{R}^m)}\rightarrow 0 \quad \text{as} \quad k\rightarrow\infty,$$ and
		\begin{equation}\label{eq5.2}
			H(\mathbf{v},{s_k})> (1-\kappa_k) H(\mathbf{c}_k,{s_k})+\kappa_k M(\mathbf{h}_k) \quad \text{for}\  \quad \mathbf{v}\in \mathbf{c}_k+W^{1,2}_0(B_1;\mathbb{R}^m),
		\end{equation} where $0<s_k<\delta_k^2$.
		
		For simplicity, rotating in $\mathbb{R}^n$ and $\mathbb{R}^m$ if necessary, we may assume that
		$$\mathbf{h}_k =\frac{f(0)\max (x_n,0)^2}{2}\mathbf{e}^1=:\mathbf{h},$$
		where $\mathbf{e}^1=(1,0,...,0)\in \mathbb{R}^m$. Hence, $\displaystyle\chi_{\{|\mathbf{h}|>0\}}=\chi_{\{x_n>0\}}$.

		From \eqref{eq5.2}, it is easy to check that
		\begin{equation}\label{eq5.3}
			(1-\kappa_k) \big(H(\mathbf{c}_k,{s_k})-M(\mathbf{h})\big)<H(\mathbf{v},{s_k})-M(\mathbf{h}),
		\end{equation}
		for every $\mathbf{v}\in W^{1,2}(B_1;\mathbb{R}^m)$ and $\mathbf{v}=\mathbf{c}_k$ on $\partial B_1$.
		
		Since $\displaystyle\Delta \mathbf{h}=\frac{f(0)\mathbf{h}}{|\mathbf{h}|}\chi_{|\mathbf{h}|>0},$ one gets
		\begin{equation}
			\int_{B_1}-\nabla \mathbf{h}\cdot\nabla\boldsymbol{\phi} \ dx+\int_{\partial B_1}\Big(\nabla\mathbf{h}\cdot \nu  \cdot\boldsymbol{\phi}-\frac{F'(0)\mathbf{h}\cdot \boldsymbol{\phi}}{2|\mathbf{h}|}\chi_{|\mathbf{h}|>0} \Big)\ d\mathcal{H}^{n-1}=0, \quad \text{for any }\quad  \boldsymbol{\phi}\in W^{1,2}(B_1;\mathbb{R}^m).\nonumber
		\end{equation} For simplicity, we denote
		\begin{equation*}
			\mathnormal{I}(\mathbf{z}):= 2\int_{B_1}\nabla\mathbf{h}\cdot \nabla(\mathbf{z}-\mathbf{h})\ dx-2\int_{\partial B_1}2\mathbf{h}\cdot(\mathbf{z}-\mathbf{h}) \ d\mathcal{H}^{n-1} +\int_{B_1}F'(0)\mathbf{e}^1\chi_{\{x_n >0\}}(\mathbf{z}-\mathbf{h})\  dx,
		\end{equation*}
		for any $\mathbf{z}\in W^{1,2}(B_1;\mathbb{R}^m)$.
		In view of the homogeneity of $\mathbf{h}$, one obtains that $I(\mathbf{z})\equiv 0$ for any $\mathbf{z}\in W^{1,2}(B_1;\mathbb{R}^m)$,
		which together with the inequality \eqref{eq5.3}, gives
		\begin{equation}
			(1-\kappa_k)\left[H(\mathbf{c}_k,{s_k})-M(\mathbf{h})-\mathnormal{I}(\mathbf{c}_k)\right]<H(\mathbf{v},{s_k})-M(\mathbf{h})-\mathnormal{I}(\mathbf{v}).\nonumber
		\end{equation}
		Thus,
		\begin{equation*}
			\begin{aligned}
				&(1-\kappa_k)\Bigg[\int_{B_1}\big(|\nabla\mathbf{c}_k|^2+\frac{1}{s_k^2}F(s_k^2|\mathbf{c}_k|)\big)\ dx-2\int_{\partial B_1}|\mathbf{c}_k|^2 d\mathcal{H}^{n-1}-\int_{B_1}\left(|\nabla \mathbf{h}|^2+F'(0)|\mathbf{h}|\right)dx \\
				&+2\int_{\partial B_1}|\mathbf{h}|^2 d\mathcal{H}^{n-1}-\int_{B_1} \big(2\nabla\mathbf{h}\cdot \nabla(\mathbf{c}_k-\mathbf{h})+{F'(0)}\chi_{\{x_n>0\}}\mathbf{e}\cdot (\mathbf{c}_k-\mathbf{h})\big) dx\Bigg] +2\int_{\partial B_1}2\mathbf{h}\cdot(\mathbf{c}_k-\mathbf{h}) d\mathcal{H}^{n-1} \\
				<&\int_{B_1}\big(|\nabla\mathbf{v}|^2+\frac{1}{s_k^2}F(s_k^2|\mathbf{v}|)\big)dx -2\int_{\partial B_1}|\mathbf{v}|^2 d\mathcal{H}^{n-1}- \int_{B_1}\big(|\nabla \mathbf{h}|^2+F'(0)|\mathbf{h}|\big)dx +2\int_{\partial B_1}|\mathbf{h}|^2 d\mathcal{H}^{n-1}\\
				&-\int_{B_1}\big(2\nabla\mathbf{h}\cdot\nabla(\mathbf{v}-\mathbf{h})+{F'(0)}\chi_{\{x_n>0\}}\mathbf{e}(\mathbf{v}-\mathbf{h})\big)dx+2\int_{\partial B_1}2\mathbf{h}\cdot(\mathbf{v}-\mathbf{h}) d\mathcal{H}^{n-1}.
			\end{aligned}
		\end{equation*}
		This yields that
		\begin{equation}\label{eq5.4}
			\begin{aligned}
				&(1-\kappa_k)\left[ \int_{B_1}|\nabla(\mathbf{c}_k-\mathbf{h})|^2 \ dx-2\int_{\partial B_1}|\mathbf{c}_k-\mathbf{h}|^2\  d\mathcal{H}^{n-1} +\int_{B_1}\frac{1}{s_k^2}F(s_k^2|\mathbf{c}_k|)\ dx-\int_{B_1^+}F'(0)\mathbf{e}\cdot\mathbf{c}_k \ dx \right] \\
				<&\int_{B_1}|\nabla(\mathbf{v}-\mathbf{h})|^2\ dx-2\int_{\partial B_1}|\mathbf{v}-\mathbf{h}|^2 \ d\mathcal{H}^{n-1}+\int_{B_1}\frac{1}{s_k^2}F(s_k^2|\mathbf{v}|)dx-\int_{B_1^+}F'(0)\mathbf{e}\cdot\mathbf{v} dx.
			\end{aligned}
		\end{equation}
		where $B_1^+:=\{x\in B_1(0): x_n>0\}$.
		
		Let $\mathbf{w}_k:=(\mathbf{c}_k-\mathbf{h})/{\delta_k}$, then $\|\mathbf{w}_k\|_{W^{1,2}(B_1,\mathbb{R}^m)}=1$. Next, it suffices to show that
		$\mathbf{w}_k\rightarrow \mathbf{w}$ strongly in $W^{1,2}(B_1;\mathbb{R}^m)$ and $\mathbf{w}= 0$ in $B_1(0)$, which leads to a contraction.
		
		Since $\|\mathbf{w}_k\|_{W^{1,2}(B_1;\mathbb{R}^m)}=1$, i.e. $\mathbf{w}_k$ is bounded in $W^{1,2}(B_1;\mathbb{R}^m)$, then there is a weakly convergent subsequence, still denoted by $\mathbf{w}_k$, that is, $\mathbf{w}_k\rightharpoonup\mathbf{w}$ in $W^{1,2}(B_1;\mathbb{R}^m)$ for some $\mathbf{w}$ in $W^{1,2}(B_1;\mathbb{R}^m)$. Therefore, we will prove $\mathbf{w}_k\to\mathbf{w}$ in $W^{1,2}(B_1;\mathbb{R}^m)$, and $\mathbf{w}\equiv 0$ in $B_1(0)$. Our proof consists of the establishment of the following four claims.

		Claim 1: $\mathbf{w}\equiv 0$ in $B_1^-$ and $\displaystyle\int_{B_1^+}(|\mathbf{c}_k|-\mathbf{e}\mathbf{c}_k)\  dx\leq C\delta_k^2$, where $B_1^-:=\{x\in B_1(0): x_n<0\}$.
		
		To see this, let $\mathbf{v}:=(1-\xi)\mathbf{c}_k+\xi\mathbf{h}$ in \eqref{eq5.4}, where $\xi(x)=\xi(|x|) \in C_0^\infty(B_1)$, $0\leq\xi \leq 1$ on $B_1$ and $\displaystyle \int_{B_1} \xi dx=1$. Since $\displaystyle\frac{\mathbf{v}-\mathbf{h}}{\delta_k}=(1-\xi)(\displaystyle\frac{\mathbf{c}_k-\mathbf{h}}{\delta_k})=(1-\xi)\mathbf{w}_k,$ then the inequality \eqref{eq5.4} implies
		\begin{equation}\label{eq5.5}
			\begin{aligned}
				&(1-\kappa_k)\Bigg[\int_{B_1}|\nabla\mathbf{w}_k|^2 dx-2\int_{\partial B_1}|\mathbf{w}_k|^2 d\mathcal{H}^{n-1}+\int_{B_1}\frac{1}{s_k^2\delta_k^2}F(s_k^2|\mathbf{c}_k|)dx-\frac{1}{\delta_k^2}\int_{B_1^+}F'(0)\mathbf{e}\cdot\mathbf{c}_k dx\Bigg]\\
				<&\int_{B_1}|\nabla((1-\xi)\mathbf{w}_k)|^2\ dx-2\int_{\partial B_1}|(1-\xi)\mathbf{w}_k|^2 d\mathcal{H}^{n-1}+\int_{B_1}\frac{1}{s_k^2\delta_k^2}F(s_k^2|\mathbf{v}|)dx-\frac{1}{\delta_k^2}\int_{B_1^+}F'(0)\mathbf{e}\cdot\mathbf{v} \ dx,
			\end{aligned}
		\end{equation}
		and  \begin{eqnarray*}
			&&\int_{B_1}|\nabla((1-\xi)\mathbf{w}_k)|^2 dx-2\int_{\partial B_1}|(1-\xi)\mathbf{w}_k|^2 \ d\mathcal{H}^{n-1}\\
			&&-(1-\kappa_k)\left(\int_{B_1}|\nabla\mathbf{w}_k|^2\ dx-2\int_{\partial B_1}|\mathbf{w}_k|^2 d\mathcal{H}^{n-1} \right)\\
			&\leq &\int_{B_1}|\nabla((1-\xi)\mathbf{w}_k)|^2 dx+2(1-\kappa_k)\int_{\partial B_1}|\mathbf{w}_k|^2d\mathcal{H}^{n-1}\\
			&\leq&C_1,
		\end{eqnarray*}
		where $C_1$ is a positive uniform constant, due to $\|\mathbf{w}_k\|_{ {W^{1,2}}(B_1;\mathbb{R}^m)}=1$.
		Then one can see that
		\begin{equation}\label{eq5.6}
			\begin{aligned}
				&\frac{1}{s_k^2\delta_k^2}\int_{B_1^-}(1-\kappa_k) F(s_k^2|\mathbf{c}_k|)-F(s_k^2|\mathbf{v}|)\ dx \\
				&+\frac{1}{\delta_k^2}\int_{B_1^+}(1-\kappa_k)\left(\frac{1}{s_k} F(s_k^2|\mathbf{c}_k|)-F'(0)\mathbf{e}\mathbf{c}_k\right)
				-\left(\frac{1}{s_k^2}F(s_k^2|\mathbf{v}|)-F'(0)\mathbf{e}\mathbf{v}\right) \ dx\\
				&\leq \ C_1.
			\end{aligned}
		\end{equation}
		
		We denote the two terms of the left side of the inequality above by $I$ and $II$, respectively. Then the convexity of $F$, taking $\mathbf{v}:=(1-\xi)\mathbf{c}_k$ in $B_1^-$ and $\mathbf{v}:=(1-\xi)\mathbf{c}_k+\xi\mathbf{h}$ in $B_1^+$ give that
		\begin{equation*}
			\begin{aligned}
				I:=&\frac{1}{s_k^2\delta_k^2}\int_{B_1^-}(1-\kappa_k) F(s_k^2|\mathbf{c}_k|)-F(s_k^2|\mathbf{v}|)\ dx\\
				\geq&\frac{1}{s_k^2\delta_k^2}\int_{B_1^-}F'(s_k^2|\mathbf{v}|)(s_k^2|\mathbf{c}_k|-s_k^2|\mathbf{v}|)-\kappa_k F'(s_k^2|\mathbf{c}_k|)s_k^2|\mathbf{c}_k| \ dx\\
				=&\frac{1}{ \delta_k^2}\int_{B_1^-}F'(s_k^2|\mathbf{v}|)( |\mathbf{c}_k|- |(1-\xi)\mathbf{c}_k|)-\kappa_k F'(s_k^2|\mathbf{c}_k|) |\mathbf{c}_k| \ dx \\
				\geq&\frac{1}{ \delta_k^2}\int_{B_1^-}\xi F'(0) |\mathbf{c}_k| -\kappa_k F'(s_k^2|\mathbf{c}_k|) |\mathbf{c}_k|\ dx.
			\end{aligned}
		\end{equation*}
		
		As for the term $II$, we obtain
		\begin{equation*}
			\begin{aligned}
				II:=&\frac{1}{\delta_k^2}\int_{B_1^+}(1-\kappa_k)\left(\frac{1}{s_k^2} F(s_k^2|\mathbf{c}_k|)-F'(0)\mathbf{e}\mathbf{c}_k\right)
				-\left(\frac{1}{s_k^2}F(s_k^2|\mathbf{v}|)-F'(0)\mathbf{e}\mathbf{v}\right) \ dx \\
				\geq& \frac{1}{\delta_k^2}\int_{B_1^+}F'(s_k^2|\mathbf{v}|)(|\mathbf{c}_k|-|\mathbf{v}|)-\kappa_k\frac{1}{s_k^2} F(s_k^2|\mathbf{c}_k|)-\left((1-\kappa_k)F'(0)\mathbf{e}\mathbf{c}_k-F'(0)\mathbf{e}\mathbf{v}\right) \ dx \\
				\geq& \frac{1}{\delta_k^2}\int_{B_1^+}F'(s_k^2|\mathbf{v}|)\xi(|\mathbf{c}_k|-|\mathbf{h}|)-\kappa_k\frac{1}{s_k^2} F(s_k^2|\mathbf{c}_k|)-F'(0)\mathbf{e}((\xi-\kappa_k)\mathbf{c}_k-\xi\mathbf{h}) \ dx\\
				\geq& \frac{1}{\delta_k^2}\int_{B_1^+}F'(s_k^2|\mathbf{v}|)\xi(|\mathbf{c}_k|-|\mathbf{h}|)-\kappa_k F'(s_k^2|\mathbf{c}_k|)|\mathbf{c}_k|-F'(0)\mathbf{e}((\xi-\kappa_k)\mathbf{c}_k-\xi\mathbf{h})\ dx.
			\end{aligned}
		\end{equation*}

		Therefore, applying the two estimates for the terms $I$ and $II$ to \eqref{eq5.6}, we get
		\begin{equation}\label{eq5.7}
			\begin{aligned}
				&\frac{1}{ \delta_k^2}\int_{B_1^-}\left(\xi F'(0)  -\kappa_k F'(s_k^2|\mathbf{c}_k|) \right)|\mathbf{c}_k| \ dx+ \frac{1}{\delta_k^2}\int_{B_1^+}\left(F'(s_k^2|\mathbf{v}|)\xi-\kappa_k F'(s_k^2|\mathbf{c}_k|)\right)|\mathbf{c}_k| -F'(0)(\xi-\kappa_k)\mathbf{e}\cdot\mathbf{c}_k \ dx\\
				\leq& \ C_1+ \frac{1}{\delta_k^2}\int_{B_1^+}F'(s_k^2|\mathbf{v}|)\xi|\mathbf{h}|-F'(0)\xi\mathbf{e}\cdot\mathbf{h} \ dx.
			\end{aligned}
		\end{equation}
		Next one shall give a lower bound for the left side of \eqref{eq5.7}. The homogeneity of $\mathbf{c}_k$ implies that, for large $k$, $\kappa_k\to 0$ and
		\begin{equation}\label{eq5.8}
			\begin{aligned}
				&\int_{B_1^-}\left(\xi F'(0)  -\kappa_k F'(s_k^2|\mathbf{c}_k|) \right)|\mathbf{c}_k| \ dx \\
				=&\int_0^1\int_{\partial B_1 \cap \{x_n<0\}}\left(\xi(\rho)F'(0)-\kappa_k F'(s_k^2\rho^2|\mathbf{c_k(\theta)}|)\right)\rho^2 \rho^{n-1}|\mathbf{c}_k(\theta)|d\mathcal{H}^{n-1} d\rho\\
				\geq&\int_0^1\int_{\partial B_1 \cap \{x_n<0\}}\left(\xi(\rho)F'(0)-\kappa_k F'(s_k^2|\mathbf{c_k(\theta)}|)\right)\rho^{n+1}|\mathbf{c}_k(\theta)|d\mathcal{H}^{n-1} d\rho\\
				\geq&\int_0^1\left(\xi(\rho)F'(0)- \frac{1}{2}\right)\rho^{n+1} d\rho \int_{\partial B_1 \cap \{x_n<0\}}|\mathbf{c}_k(\theta)|d\mathcal{H}^{n-1}\\
				\geq&c_1\int_{\partial B_1 \cap \{x_n<0\}} |\mathbf{c}_k|d\mathcal{H}^{n-1},
			\end{aligned}
		\end{equation}
		where $c_1>0$ depends only on $c_0$, $\xi$ and $n$.
		
		On the other hand, for the term
		$$\int_{B_1^+}\left(F'(s_k^2|\mathbf{v}|)\xi-\kappa_k F'(s_k^2|\mathbf{c}_k|)\right)|\mathbf{c}_k|-F'(0)(\xi-\kappa_k)\mathbf{e}\cdot\mathbf{c}_k\ dx,$$
		it is easy to see that
		\begin{equation}\label{eq5.9}
			\begin{aligned}
				&\int_{B_1^+}\left(F'(0)\xi-\kappa_k F'(0)\right)|\mathbf{c}_k|-F'(0)(\xi-\kappa_k)\mathbf{e}\cdot\mathbf{c}_k \ dx\\
				=&\int_{B_1^+}F'(0)\left(\xi-\kappa_k \right)\left(|\mathbf{c}_k|-\mathbf{e}\cdot\mathbf{c}_k\right) \ dx\\
				=&\int_0^1\int_{\partial B_1 \cap \{x_n>0\}}F'(0)\left(\xi(\rho)-\kappa_k \right)\left(\rho^2|\mathbf{c}_k(\theta)|-\rho^2\mathbf{e}\cdot\mathbf{c}_k(\theta)\right)\rho^{n-1}d\mathcal{H}^{n-1} d\rho\\
				=&\int_0^1F'(0)\left(\xi(\rho)-\kappa_k \right)\rho^{n+1}d\rho \int_{\partial B_1 \cap \{x_n>0\}}\left(|\mathbf{c}_k(\theta)|-\mathbf{e}\cdot\mathbf{c}_k(\theta)\right)d\mathcal{H}^{n-1}\\
				\geq&c_2\int_{\partial B_1 \cap \{x_n>0\}}\left(|\mathbf{c}_k(\theta)|-\mathbf{e}\cdot\mathbf{c}_k(\theta)\right)d\mathcal{H}^{n-1},
			\end{aligned}
		\end{equation}
		where $c_2>0$ depends only on $\xi$ and $n$.
		Therefore, the inequalities \eqref{eq5.7}-\eqref{eq5.9} lead to that
		\begin{equation*}
			\begin{aligned}
				&\frac{c_1}{\delta_k^2}\int_{\partial B_1 \cap \{x_n<0\}} |\mathbf{c}_k|d\mathcal{H}^{n-1}+\frac{c_2}{\delta_k^2}\int_{\partial B_1 \cap \{x_n>0\}}\left(|\mathbf{c}_k(\theta)|-\mathbf{e}\cdot\mathbf{c}_k(\theta)\right)d\mathcal{H}^{n-1}\\
				\leq& C_1+ \frac{1}{\delta_k^2}\int_{B_1^+}\big(F'(s_k^2|\mathbf{v}|)|\mathbf{h}|-F'(0)\mathbf{e}\cdot\mathbf{h}\big)\xi \ dx\\
				=& C_1+ \frac{1}{\delta_k^2}\int_{B_1^+}F''(0)s_k^2|\mathbf{v}||\mathbf{h}|\xi+o(s_k^2) \ dx \quad \text{as} \quad k\to \infty.
			\end{aligned}
		\end{equation*}
		This shows
		\begin{equation*}
			\int_{B_1^-}|\mathbf{c}_k|\ dx \leq C_2\delta_k^2\quad \text {and} \quad \int_{B_1^+}(|\mathbf{c}_k|-\mathbf{e^1\cdot\mathbf{c}_k})\ dx\leq C_2\delta_k^2.
		\end{equation*} In particular,
		\begin{equation*}
			\int_{B_1^-}|\mathbf{w}_k| \ dx= \int_{B_1^-}\left|\frac{\mathbf{c}_k-\mathbf{h}}{\delta_k}\right| \ dx= \int_{B_1^-}\left|\frac{\mathbf{c}_k}{\delta_k}\right| \ dx=\frac{1}{\delta_k}\int_{B_1^-}|\mathbf{c}_k| \ dx \leq C_2\delta_k,
		\end{equation*}
		i.e. $\mathbf{w}\equiv0$ in $B_1^-$. Therefore Claim 1 holds true.\\

		Claim 2: $\Delta(\mathbf{e}\cdot\mathbf{w})=0$ in $B_1^+$, $\mathbf{e^j}\cdot\mathbf{w}=d_jh$ in $B_1^+(0)$ for each $j>1$ and some constant $d_j$, where $h:=\max(x_n,0)^2/2$, namely,  $\mathbf{h}=h\cdot\mathbf{e}^1$.
		
		To prove it, one can fix a ball $B_0\Subset B_1^+(0)$, and let $\mathbf{v}:=(1-\eta)\mathbf{c}_k+ \eta(\mathbf{h}+\delta_k\mathbf{g})$ in \eqref{eq5.4}, where $\eta\in C_0^{\infty}(B_1^+)$ and $\mathbf{g}\in W^{1,2}(B_1;\mathbb{R}^m)$ such that $\eta=1$ in $B_0$, $\eta=0$ in $B_1^-$, $0\leq\eta\leq1$ in $B_1^+\setminus B_0$.  Then rescaling the inequality \eqref{eq5.4}, we arrive at
		\begin{equation*}
			\begin{aligned}
				&(1-\kappa_k)\Big(\int_{B_1}|\nabla\mathbf{w}_k|^2\ dx -2\int_{\partial B_1}|\mathbf{w}_k|^2 \ d \mathcal{H}^{n-1}+\int_{B_1}\frac{1}{s_k^2\delta_k^2}F(s_k^2 |\mathbf{c}_k|) dx -\frac{1}{\delta_k^2}\int_{B_1^+}F'(0)\mathbf{e}\cdot\mathbf{c}_k \ dx\Big)\\
				<&\frac{1}{\delta_k^2}\int_{B_1}|\nabla(\mathbf{v}-\mathbf{h})|^2\ dx-\frac{2}{\delta_k^2}\int_{\partial B_1}|\mathbf{v}-\mathbf{h}|^2\ d \mathcal{H}^{n-1} +\int_{B_1}\frac{1}{s_k^2\delta_k^2}F(s_k^2|\mathbf{v}|)dx-\frac{1}{\delta_k^2}\int_{B_1^+}F'(0)\mathbf{e}\cdot\mathbf{v} dx \\
				< &\int_{B_1}|\nabla((1-\eta)\mathbf{w}_k+\eta\mathbf{g})|^2 \ dx -2\int_{\partial B_1}|(1-\eta)\mathbf{w}_k+\eta\mathbf{g}|^2 \ d\mathcal{H}^{n-1}\\
				&+ \ \int_{B_1}\frac{1}{s_k^2\delta_k^2}F(s_k^2|(1-\eta)\mathbf{c}_k+\eta(\mathbf{h}+\delta_k\mathbf{g})|)\ dx-\frac{1}{\delta_k^2}\int_{B_1^+}F'(0)\mathbf{e}\cdot((1-\eta)\mathbf{c}_k+\eta(\mathbf{h}+\delta_k\mathbf{g}))\ dx,
			\end{aligned}
		\end{equation*}
		which implies
		\begin{equation*}
			\begin{aligned}
				&(2\eta-\eta^2)\int_{B_1^+}|\nabla\mathbf{w}_k|^2\ dx +\int_{B_1^+}\frac{1}{s_k^2\delta_k^2}\left(F(s_k^2|\mathbf{c}_k|)- s_k^2 F'(0)\mathbf{e}\cdot\mathbf{c}_k\right) dx \\
				\leq&\ o(1) +\int_{B_1^+}\eta^2|\nabla \mathbf{g}|^2\ dx+\int_{B_1^+\setminus B_0}\Big(|\nabla\eta|^2(\mathbf{w}_k-\mathbf{g})^2+(2\eta-2\eta^2)\nabla\mathbf{w}_k\cdot\nabla\mathbf{g}\Big)\ dx\\
				&+\int_{B_1^+\setminus B_0}\Big(2(1-\eta)\nabla\eta\cdot\nabla\mathbf{w}_k\cdot(\mathbf{g}-\mathbf{w}_k) +2\eta\nabla\eta\cdot \nabla\mathbf{g}\cdot(\mathbf{g}-\mathbf{w}_k)\Big)\ dx\\
				&+ \ \frac{1}{\delta_k^2}\int_{B_1^+}\left(\frac{1}{s_k^2}F(s_k^2|(1-\eta)\mathbf{c}_k+\eta(\mathbf{h}+\delta_k\mathbf{g})|-F'(0)\mathbf{e}\cdot ((1-\eta)\mathbf{c}_k+\eta(\mathbf{h}+\delta_k\mathbf{g}))\right) \ dx.
			\end{aligned}
		\end{equation*}
		Then the boundedness of $F''$ give that
		\begin{equation*}
			\begin{aligned}
				\int_{B_1^+}\frac{1}{s_k^2\delta_k^2}\left(F(s_k^2|\mathbf{c}_k|)- s_k^2 F'(0)\mathbf{e}\cdot\mathbf{c}_k\right) dx=\frac{1}{\delta_k^2}\int_{B_1^+}F'(0)\eta\left(|\mathbf{c}_k|-\mathbf{e}\cdot\mathbf{c}_k)\right) \ dx+\ o(1).
			\end{aligned}
		\end{equation*}
		Through the above equality, we obtain that
		\begin{equation}\label{5.10}
			\begin{aligned}
				&\int_{B_1^+}(2\eta-\eta^2)|\nabla\mathbf{w}_k|^2 \ dx+\ \frac{1}{\delta_k^2}\int_{B_1^+}F'(0)\eta\left(|\mathbf{c}_k|-\mathbf{e}\cdot\mathbf{c}_k)\right) \ dx \\
				\leq& \ o(1)+\int_{B_1^+} \Big(\eta^2|\nabla\mathbf{g}|^2 \ dx +\int_{B_1^+\setminus B_0}|\nabla\eta|^2(\mathbf{w}_k-\mathbf{g})^2+(2\eta-2\eta^2)\nabla\mathbf{w}_k\cdot\nabla\mathbf{g} \Big)\ dx\\
				&+\ \int_{B_1^+\setminus B_0}\Big(2(1-\eta)\nabla\eta\cdot\nabla\mathbf{w}_k\cdot(\mathbf{g}-\mathbf{w}_k) +2\eta\nabla\eta\cdot \nabla\mathbf{g}\cdot(\mathbf{g}-\mathbf{w}_k)\Big) \ dx\\
				&+\ \frac{1}{\delta_k^2}\int_{B_1^+}F'(0)\eta \left(|(\mathbf{h}+\delta_k\mathbf{g})|- \mathbf{e}\cdot(\mathbf{h}+\delta_k\mathbf{g})\right)\ dx.
			\end{aligned}
		\end{equation}
		
		From the direct computation, it implies that
		\begin{equation}\label{eq5.11}
			|\mathbf{h}+\delta_k\mathbf{g}|-\mathbf{e}\cdot(\mathbf{h}+\delta_k\mathbf{g})=\frac{1}{2}\delta_k^2\frac{|\mathbf{g}|^2-(\mathbf{e}\mathbf{g})^2}{(h+\delta_k\mathbf{e}\mathbf{g})^2}+o(\delta_k^2),
		\end{equation}
		and
		\begin{equation}\label{eq5.12}
			|\mathbf{c}_k|-\mathbf{e}\cdot\mathbf{c}_k=\frac{1}{2}\delta_k^2\frac{|\mathbf{w}_k|^2-(\mathbf{e}\mathbf{w}_k)^2}{(h+\delta_k\mathbf{e}\mathbf{w}_k)^2}+o(\delta_k^2).
		\end{equation}

		Let $\mathbf{g}$ be any function in $W^{1,2}(B_1;\mathbb{R}^m)$ satisfying $\mathbf{g}=\mathbf{w}$ in $B_1\setminus B_0$. Inserting \eqref{eq5.11} and \eqref{eq5.12} into (\ref{5.10}) and taking $k\rightarrow\infty$, we deduce
		\begin{equation}\label{5.12}
			\begin{aligned}
				&\int_{B_0}|\nabla\mathbf{w}|^2dx+\int_{B_0}F'(0)h\left(\frac{1}{2}\frac{|\mathbf{w}|^2-(\mathbf{e}\mathbf{w})^2}{h^2}\right)dx\\
				\leq& \int_{B_0}|\nabla\mathbf{g}|^2\ dx+ \int_{B_0}F'(0)h\left(\frac{1}{2}\frac{|\mathbf{g}|^2-(\mathbf{e}\mathbf{g})^2}{h^2}\right)\ dx.
			\end{aligned}
		\end{equation}
		In \eqref{5.12}, for $\mathbf{w}=(w_1,w_2,w_3,...,w_m)$, one may take $\mathbf{g}=(g_1,w_2,w_3,...,w_m)$ with $g_1\in W^{1,2}(B_1)$ and $g_1=w_1$ in $B_1\setminus B_0$, then it infers
		$$\int_{B_0}|\nabla w_1|^2 dx\leq \int_{B_0}|\nabla g_1|^2 dx$$
		i.e. $\Delta(\mathbf{e}\cdot\mathbf{w})=0$.  In addition one also can take $\mathbf{g}=(w_1,w_2,,...g_j...,w_m)$ for any $2\leq j\leq m$ with $g_j=w_j$ in $B_1\setminus {B_0}$,  then one obtains
		$$\int_{B_0}|\nabla w_j|^2 dx+\int_{B_0}\frac{F'(0)}{2h}w_j^2 dx \leq \int_{B_0}|\nabla g_j|^2 dx +\int_{B_0} \frac{F'(0)}{2h}g_j^2 dx,$$
		which shows $\Delta(\mathbf{e^j\cdot\mathbf{w}})=\displaystyle\frac{F'(0)}{2h}\mathbf{e^j}\cdot\mathbf{w}$ in $B_0$ for $j>1$ for real number $d_j$.
		According to Lemma \ref{laplace}, it follows that $\mathbf{e^j}\cdot\mathbf{w}(x)=d_jh(x)$ for some real number $d_j$.

		Claim 3: $\mathbf{e}\cdot \mathbf{w}=0$ in $B_1(0)$ and $d_j=0$ for each $j\geq2$.
		
		For the proof of Claim 3, the process is essentially same as that of \cite[Theorem 1]{asu}. We will briefly describe the proof process.
		
		Based on that $\mathbf{e}^1\mathbf{w}$ is homogeneous harmonic function of degree 2 in $B_1^+$ and $\mathbf{e}^1\mathbf{w}=0$ in $B_1^-$. Firstly, we can obtain $\mathbf{e}\cdot \mathbf{w}=\sum_{j=1}^{n-1}a_{nj}x_jx_n$ in $B_1$ through odd reflection and {\it Liouville Theorem}. Then recalling the selection of $\mathbf{h}$ which is the minimizer of $\inf_{h\in\mathbf{H}}\|\mathbf{c}_k-\mathbf{h}\|_{W^{1,2}(B_1;\mathbb{R}^m)}$, it gives that $\mathbf{e}^1\mathbf{w}=0$ in $B_1$.
		
		Secondly, Claim 2 has showed $\mathbf{e^j}\cdot\mathbf{w}=d_jh$ in $B_1^+(0)$ for each $j>1$ and some constant $d_j$, thus we can denote $\mathbf{w}_k=\mathbf{d}h+\mathbf{z}_k$, where $\mathbf{d}=(d_1,d_2,...,d_m)$, $\mathbf{d}\mathbf{e}^1=0$ and $\mathbf{z}_k\to 0$ weakly in $W^{1,2}(B_1;\mathbb{R}^m)$ as $k\to\infty$. From the assumption
		$$1=\|\mathbf{w}_k\|^2_{W^{1,2}(B_1;\mathbb{R}^m)}=\|\mathbf{d}h+\mathbf{z}_k\|^2_{W^{1,2}(B_1;\mathbb{R}^m)},$$
		it follows that $|\mathbf{d}|=0$, Therefore, $d_j=0$ for each $j\geq2$.
		
		Claim 4: $\mathbf{w}_k\to\mathbf{w}$ strongly in $W^{1,2}(B_1;\mathbb{R}^m)$.
		
		To see this, let $\mathbf{v}=(1-\zeta)\mathbf{c}_k+\zeta\mathbf{h}$ with
		\begin{align*}
			\begin{split}
				\zeta(x)=\zeta(|x|)=\left\{
				\begin{array}{lr}
					0,                 &|x|\geq  1\\
					2-2|x|,                 &1/2<|x|\leq 1\\
					1,         &|x|\leq \frac{1}{2}\\
				\end{array}
				\right.
			\end{split}
		\end{align*}
		i.e. $\zeta(|x|)=\min\left(2\max(1-|x|,0),1\right)$.
		
		Observing $\displaystyle\frac{\mathbf{v}-\mathbf{h}}{\delta_k}=(1-\zeta)\mathbf{w}_k$,  then from the inequality \eqref{eq5.5}, one obtains
		\begin{equation*}
			\begin{aligned}
				&(1-\kappa_k)\Big(\int_{B_1}|\nabla\mathbf{w}_k|^2 dx -2\int_{\partial B_1}|\mathbf{w}_k|^2 d\mathcal{H}^{n-1}+\int_{B_1}\frac{1}{s_k^2\delta_k^2}F(s_k^2|\mathbf{c}_k|)dx-\frac{1}{\delta_k^2}\int_{B_1^+}F'(0)\mathbf{e}\cdot\mathbf{c}_k dx \Big)\\
				<&\int_{B_1}|\nabla((1-\zeta)\mathbf{w}_k)|^2 dx -2\int_{\partial B_1}|(1-\zeta)\mathbf{w}_k|^2 d\mathcal{H}^{n-1}+\int_{B_1}\frac{1}{s_k^2\delta_k^2}F(s_k^2|\mathbf{v}|)dx-\frac{1}{\delta_k^2}\int_{B_1^+}F'(0)\mathbf{e}\cdot\mathbf{v}dx\\
				=&\int_{B_1}(1-\zeta)^2|\nabla\mathbf{w}_k|^2+|\mathbf{w}_k|^2|\nabla\zeta|^2-2(1-\xi)\nabla\mathbf{w}_k\cdot\mathbf{w}_k\cdot\nabla \zeta \ dx-2\int_{\partial B_1}(1+\zeta^2-2\zeta)|\mathbf{w}_k|^2 d\mathcal{H}^{n-1}\\
				&+\int_{B_1}\frac{1}{s_k^2\delta_k^2}F(s_k^2|\mathbf{v}|)dx-\frac{1}{\delta_k^2}\int_{B_1^+}F'(0)\mathbf{e}\cdot\mathbf{v} dx.
			\end{aligned}
		\end{equation*}
		This gives
		\begin{equation*}
			\begin{aligned}
				&(2\zeta-\zeta^2)\int_{B_1}|\nabla\mathbf{w}_k|^2 dx +\int_{B_1^-}\frac{1}{s_k^2\delta_k^2}F(s_k^2|\mathbf{c}_k|)dx+\int_{B_1^+}\frac{1}{s_k^2\delta_k^2}F(s_k^2|\mathbf{c}_k|)-\frac{1}{\delta_k^2}F'(0)\mathbf{e}\cdot\mathbf{c}_k dx\\
				<&\kappa_k\left(\int_{B_1}|\nabla\mathbf{w}_k|^2 dx -2\int_{\partial B_1}|\mathbf{w}_k|^2 d\mathcal{H}^{n-1} +\int_{B_1}\frac{1}{s_k^2\delta_k^2}F(s_k^2|\mathbf{c}_k|)dx-\frac{1}{\delta_k^2}\int_{B_1^+}F'(0)\mathbf{e}\cdot\mathbf{c}_k dx \right)\\
				&+\int_{B_1}|\nabla\zeta|^2|\mathbf{w}_k|^2-2(1-\zeta)\nabla\mathbf{w}_k\cdot\mathbf{w}_k\cdot\nabla \zeta dx +\int_{B_1^-}\frac{1}{s_k^2\delta_k^2}F(s_k^2|\mathbf{v}|)dx\\
				&+\int_{B_1^+}\frac{1}{s_k^2\delta_k^2}F(s_k^2|\mathbf{v}|)-\frac{1}{\delta_k^2}F'(0)\mathbf{e}\cdot\mathbf{v} dx.
			\end{aligned}
		\end{equation*}
		Since $\mathbf{v}=(1-\xi)\mathbf{c}_k$ in $B_1^-$, then $|\mathbf{c}_k|>|\mathbf{v}| $ and it infers
		\begin{equation*}
			\int_{B_1^-}\frac{1}{s_k^2\delta_k^2}F(s_k^2|\mathbf{c}_k|)dx-\int_{B_1^-}\frac{1}{s_k^2\delta_k^2}F(s_k^2|\mathbf{v}|)dx>0.
		\end{equation*}
		Notice the fact \eqref{eq5.11}, one has
		\begin{equation*}
			\begin{aligned}
				&\int_{B_1^+}\frac{1}{s_k^2\delta_k^2}F(s_k^2|\mathbf{c}_k|)-\frac{1}{\delta_k^2}F'(0)\mathbf{e}\cdot\mathbf{c}_k dx \\
				=&\int_{B_1^+}F'(0)(h+\delta_k\mathbf{e}\mathbf{w}_k)\left(\frac{1}{2}\frac{|\mathbf{w}_k|^2-(\mathbf{e}\mathbf{w}_k)^2}{(h+\delta_k\mathbf{e}\mathbf{w}_k)^2}+o(1)\right)+o(1) dx,
			\end{aligned}
		\end{equation*}
		and
		\begin{equation*}
			\begin{aligned}
				&\int_{B_1^+}\frac{1}{s_k^2\delta_k^2}F(s_k^2|\mathbf{v}|)-\frac{1}{\delta_k^2}F'(0)\mathbf{e}\cdot\mathbf{v} dx \\
				=&\int_{B_1^+}F'(0)(h+(1-\zeta)\delta_k\mathbf{e}\mathbf{w}_k)\left(\frac{1}{2}\frac{(1-\zeta)^2|\mathbf{w}_k|^2-(1-\xi)(\mathbf{e}\mathbf{w}_k)^2}{(h+(1-\zeta)\delta_k\mathbf{e}\mathbf{w}_k)^2}+o(1)\right)+o(1) dx.
			\end{aligned}
		\end{equation*}
		
		Recalling that Claim 1-3, one has shown that as $k \to + \infty$, $\mathbf{w}_k$ converges to $\mathbf{0}$ weakly in $W^{1,2}(B_1;\mathbb{R}^m)$.
		Hence one knows
		\begin{equation*}
			\lim_{k\rightarrow\infty}\int_{B_{\frac{1}{2}}}|\nabla\mathbf{w}_k|^2 dx = 0,
		\end{equation*}
		and due to the homogeneity of $\mathbf{w}$,  it leads to
		\begin{equation*}
			\lim_{k\rightarrow\infty}\int_{B_{1}}|\nabla\mathbf{w}_k|^2 dx= 0.
		\end{equation*}
		Therefore, $\mathbf{w}_k\to\mathbf{0}$ strongly in $W^{1,2}(B_1;\mathbb{R}^m)$.
		This contradicts with the fact $\|\mathbf{w}_k\|_{W^{1,2}(B_1;\mathbb{R}^m)}=1$.

	\end{proof}

	\vspace{25pt}

	\section{An energy decay estimate and uniqueness of blow-up limit}
	
	In this section, we show that an energy decay estimate via the epiperimetric inequality and then based on the energy decay we obtain the uniqueness of blow-up limit.
	\begin{proposition}{(Energy decay and uniqueness of blow-up limit)}\label{uniqueness}
		Let $x^0\in \Gamma_0(\mathbf{u})$ and suppose that the epiperimetric inequality \eqref{6.1} holds true with $\kappa\in(0,1)$ and $r_0 \in (0,1)$, for each value of $\mathbf{c}$, taking $\mathbf{c}_r$ as follows,
		$$\mathbf{c}_r(x):=|x|^2\mathbf{u}_r(\frac{x}{|x|})=\frac{|x|^2}{r^2}\mathbf{u}(x^0+\frac{r}{|x|}x), \quad \text{for any}\quad 0<r\leq r_0<1.$$
		Assume that $\mathbf{u}_0$ denotes an arbitrary blow-up limit of $\mathbf{u}$ at $x^0$.
		Then there holds
		\begin{equation}\label{eq6.1}
			|W(\mathbf{u},x^0,r)-W(\mathbf{u},x^0,0+)|\leq |W(\mathbf{u},x^0,r_0)-W(\mathbf{u},x^0,0+)|\left(\frac{r}{r_0}\right)^{\frac{(n+2)\kappa}{1-\kappa}}, \quad \text{for $r\in(0,r_0)$},
		\end{equation}
		and there is a constant $C=C(n,\kappa)>0$ such that for $r\in(0,\frac{r_0}{2})$, we have
		\begin{equation}\label{eq6.2}
			\int_{\partial B_1(0)}\left|\frac{\mathbf{u}(x^0+rx)}{r^2}-\mathbf{u}_0(x)\right|d\mathcal{H}^{n-1}\leq C|W(\mathbf{u},x^0,r_0)-W(\mathbf{u},x^0,0+)|^{\frac{1}{2}}\left(\frac{r}{r_0}\right)^{\frac{(n+2)\kappa}{2(1-\kappa)}}.
		\end{equation}
	\end{proposition}
	\begin{remark}
		\begin{align*}
			W(\mathbf{u},x^0,r)=H(\mathbf{u}_{r},r):=\int_{B_1}|\nabla \mathbf{u}_{r}|^2+\frac{1}{r^2}F(r^2 |\mathbf{u}_{r}|)dx-2\int_{\partial B_1}|\mathbf{u}_{r}|^2d \mathcal{H}^{n-1},
		\end{align*}
		we will employ the epiperimetric inequality, 
		$$H(\mathbf{c}_r,r)-W(\mathbf{u},x^0,0^+)\ \geq \  \frac{1}{1-\kappa}\left(H(\mathbf{v},r)-W(\mathbf{u},x^0,0^+)\right).$$
		This inequality will be used to derive a linear energy inequality, which, in turn, will be utilized to achieve energy decay.
	\end{remark}

	\begin{proof}
		Define $G(r):=W(\mathbf{u},x^0,r)-W(\mathbf{u},x^0,0+)$, then $$\displaystyle G(r)=r^{-n-2}\int_{B_r(x^0)}\left(|\nabla\mathbf{u}|^2+F(|\mathbf{u}|)\right)dx-2r^{-n-3}\int_{\partial B_r(x^0)}|\mathbf{u}|^2d\mathcal{H}^{n-1}-W(\mathbf{u},x_0,0^+),$$
		and a direct computation gives that
		\begin{equation*}
			\begin{aligned}
				G'(r)=&-\frac{n+2}{r}G(r)-\frac{n+2}{r}\frac{2}{r^{n+3}}\int_{\partial B_r(x^0)}|\mathbf{u}|^2d\mathcal{H}^{n-1}-\frac{n+2}{r}W(\mathbf{u},x^0,0^+)\\
				&+r^{-n-2}\int_{\partial B_r(x^0)}\left(|\nabla\mathbf{u}|^2+F(|\mathbf{u}|)\right)d\mathcal{H}^{n-1}\\
				&+8r^{-n-4}\int_{\partial B_r(x^0)}|\mathbf{u}|^2 d\mathcal{H}^{n-1}-2r^{-n-3}\int_{\partial B_r(x^0)}2\mathbf{u}\cdot\nabla\mathbf{u}\cdot\nu d\mathcal{H}^{n-1} \\
				=&r^{-1}\Big( \int_{\partial B_1}(|\nabla\mathbf{u}_r|^2+\frac{1}{r^2}F(r^2|\mathbf{u}_r|))-(2n-4)|\mathbf{u}_r|^2 -4\nabla\mathbf{u}_r\cdot\nu \cdot\mathbf{u}_rd\mathcal{H}^{n-1}\Big)\\
				& -\frac{n+2}{r}W(\mathbf{u},x^0,0^+)  -\frac{n+2}{r}G(r).
			\end{aligned}
		\end{equation*}
		Furthermore,
		\begin{equation}\label{eq6.3}
			\begin{aligned}
				G'(r)\geq&r^{-1}\Big( \int_{\partial B_1}(|\nabla_{\theta}\mathbf{u}_r|^2+\frac{1}{r^2}F(r^2|\mathbf{u}_r|))+4|\mathbf{u}_r|^2 -2(n+2)|\mathbf{u}_r|^2 \\
				&-(n+2)W(\mathbf{u},x^0,0^+) d\mathcal{H}^{n-1} \Big)-\frac{n+2}{r}G(r)\\
				\geq&r^{-1}\Big( \int_{\partial B_1}|\nabla_{\theta}\mathbf{c}_r|^2+|\nabla\mathbf{c}_r(x)\cdot\nu|^2+\frac{1}{r^2}F(r^2|\mathbf{c}_r|)-2(n+2)|\mathbf{c}_r|^2 d\mathcal{H}^{n-1}  \Big)\\
				&-\frac{n+2}{r}W(\mathbf{u},x^0,0^+) -\frac{n+2}{r} G(r).
			\end{aligned}
		\end{equation}
		Since $\displaystyle\frac{1}{r^2}F(r^2s)$ is an increasing function with respect to $r$ for any fixed $s > 0$, one can know that
		\begin{equation*}
			\begin{aligned}
				\int_{B_1}|\nabla\mathbf{c}_r(x)|^2+\frac{1}{r^2}F(r^2|\mathbf{c}_r(x)|) dx=&\int_0^1\int_{\partial B_{1}(0)}\rho^{n-1}\left(|\rho^2\nabla \mathbf{c}_r(y)|^2+\frac{1}{r^2}F(r^2\rho^2|\mathbf{c}_r(y)|)\right)d\mathcal{H}^{n-1}d\rho\\
				\leq&\int_0^1\int_{\partial B_{1}(0)}\rho^{n+1}\left(|\nabla \mathbf{c}_r(y)|^2+\frac{1}{\rho^2r^2}F(r^2\rho^2|\mathbf{c}_r(y)|)\right)d\mathcal{H}^{n-1}d\rho\\
				\leq&\int_0^1\int_{\partial B_{1}(0)}\rho^{n+1}\left(|\nabla \mathbf{c}_r(y)|^2+\frac{1}{r^2}F(r^2|\mathbf{c}_r(y)|)\right)d\mathcal{H}^{n-1}d\rho\\
				=&\frac{1}{n+2}\int_{\partial B_{1}(0)}\left(|\nabla \mathbf{c}_r|^2+\frac{1}{r^2}F(r^2|\mathbf{c}_r|)\right)d\mathcal{H}^{n-1}.
			\end{aligned}
		\end{equation*}which together with \eqref{eq6.3},  implies
		\begin{equation}\label{eq6.4}
			G'(r)\geq\frac{n+2}{r}\Big(H(\mathbf{c}_r,r)-W(\mathbf{u},x^0,0^+)-G(r)\Big).
		\end{equation}
		Notice the assumptions that the epiperimetric inequality holds true for each $\mathbf{c}_r(x)$, i.e. one can see that $$H(\mathbf{v},s)\leq(1-\kappa)H(\mathbf{c}_r,s)+\kappa W(\mathbf{u},x^0,0^+),$$ for any $s\in(0,r_0)$ and some $\mathbf{v}\in W^{1,2}(B_1;\mathbb{R})^m)$ with $\mathbf{v}=\mathbf{c}_r$ on $\partial B_1$. Hence, we take $s=r\in(0,r_0),$ into above inequality, then which together with \eqref{eq6.3}-\eqref{eq6.4} shows that
		\begin{eqnarray}\label{eq6.5}
			G'(r)&\geq&\frac{n+2}{r}\frac{1}{1-\kappa}(H(\mathbf{v},r)-W(\mathbf{u},x^0,0^+))-\frac{n+2}{r}G(r)\nonumber\\
			&\geq&\frac{n+2}{r}\frac{1}{1-\kappa}(H(\mathbf{u}_r,r)-W(\mathbf{u},x^0,0^+))-\frac{n+2}{r}G(r)\nonumber\\
			&=&\frac{(n+2)\kappa}{(1-\kappa)r}G(r).
		\end{eqnarray}Here we have used the fact that $\mathbf{u}$ is the minimizer of the problem (1.1) with $\mathbf{u}_r =\mathbf{c}_r$ on $\partial B_1$. According to the monotonicity formula in Lemma \ref{monotonicity}, $G(r)\geq 0$ for any $r\in(0, r_0)$ and we conclude in the non-trivial case $G(r)>0$ in $(r_1,r_0)$ for any small $r_1>0$.  Thus the estimate from \eqref{eq6.5} implies
		\begin{equation}\label{eq6.6}
			\begin{aligned}
				\frac{G(r_0)}{G(r)}&\geq  \left(\frac{r_0}{r}\right)^{\frac{(n+2)\kappa}{(1-\kappa)}}, \quad \text{for} \quad r\in(r_1,r_0),
			\end{aligned}
		\end{equation}
		and for $0<\rho<\sigma\leq r_0$, we get
		\begin{equation}\label{eq6.7}
			\begin{aligned}
				&\int_{\partial B_1(0)}\int_\rho^\sigma\left|\frac{d\mathbf{u}_r}{dr}\right|drd\mathcal{H}^{n-1}\\
				=&\int_\rho^\sigma  r^{-1-n}\int_{\partial B_r(x_0)}\left|\nabla\mathbf{u}\cdot\nu -2\frac{\mathbf{u}}{r}\right|d\mathcal{H}^{n-1} dr\\
				\leq&\sqrt{\frac{n\omega_n}{2}}\int_\rho^\sigma r^{-\frac{1}{2}}\left(2 r^{-n-2}\int_{\partial B_r(x_0)}\left|\nabla\mathbf{u}\cdot\nu -2\frac{\mathbf{u}}{r}\right|^2d\mathcal{H}^{n-1}\right)^{\frac{1}{2}}dr\\
				\leq&\sqrt{\frac{n\omega_n}{2}}\int_\rho^\sigma r^{-\frac{1}{2}}\sqrt{e'(r)}dr\\
				\leq&\sqrt{\frac{n\omega_n}{2}}\left(\log (\sigma)-\log(\rho)\right)^{\frac{1}{2}} \left(e(\sigma)-e(\rho)\right)^{\frac{1}{2}},
			\end{aligned}
		\end{equation} where Lemma \ref{monotonicity} and Cauchy-Schwartz inequality have been applied. For any $0<2\rho<2r\leq r_0$, there exist two positive integers $l<m$ such that $\rho\in [2^{-m-1}, 2^{-m})$ and $r \in [2^{-l-1}, 2^{-l})$. Hence recalling that \eqref{eq6.5}-\eqref{eq6.7}, we obtain
		\begin{equation}\label{eq6.8}
			\begin{aligned}
				&\int_{\partial B_1(0)}\big|{\frac{\mathbf{u}(x_0+r x)}{r^2}-\frac{\mathbf{u}(x_0+\rho x)}{\rho^2}}\big|d\mathcal{H}^{n-1}\\
				\leq& \sum_{j=l}^m \int_{\partial B_1(0)}\int_{2^{-j-1}}^{2^{-j}}|\frac{d\mathbf{u}_r}{dr}|drd\mathcal{H}^{n-1}\\
				\leq&C(n)\sum_{j=l}^m\left(\log (2^{-j})-\log(2^{-j-1})\right)^{\frac{1}{2}} \left(e(2^{-j})-e(2^{-j-1})\right)^{\frac{1}{2}}\\
				\leq&C(n)\sum_{j=l}^m \left|W(\mathbf{u}, x_0, r_0)-W(\mathbf{u}, x_0, 0+)\right|^{\frac{1}{2}}(r_0 2^{j})^{\frac{-(n+2)\kappa}{2(1-\kappa)}}\\
				\leq&C(n)\left|W(\mathbf{u}, x_0, r_0)-W(\mathbf{u}, x_0, 0+)\right|^{\frac{1}{2}}r_0^{\frac{-(n+2)\kappa}{2(1-\kappa)}}\frac{c^l}{1-c} \\
				\leq&C(n,\kappa)\mid W(\mathbf{u}, x_0, r_0)-W(\mathbf{u}, x_0, 0+)\mid^{\frac{1}{2}}\left(\frac{r}{r_0}\right)^{\frac{(n+2)\kappa}{2(1-\kappa)}},
			\end{aligned}
		\end{equation}
		where $c= \displaystyle2^{{\frac{-(n+2)\kappa}{2(1-\kappa)}}}$.  Finally, let $\frac{\mathbf{u}(x^0+\rho_j x)}{\rho_j}\to \mathbf{u}_0$ as a certain sequence $\rho_j \rightarrow0+$, one can see \eqref{eq6.2}. Thus  we can conclude the proof.

	\end{proof}

	\vspace{15pt}

	\section{Regularity of free boundary}
	
	In this section, we will give the proof of the main theorem (Theorem \ref{regularity})  that the free boundary of an open neighbourhood of regular free boundary points $\mathcal{R}_{\mathbf{u}}$ is a $C^{1,\beta}$-surface in $D$. For this purpose, we firstly need to verify the assumption of the energy decay estimate (Proposition \ref{uniqueness}) uniformly in an open neighborhood of a regular free boundary point.
	
	\begin{lemma}\label{regularity1}
		Let $C_h$ be a compact set of points $x^0 \in \Gamma_0(\mathbf{u})$ with the following property, at least one blow-up limit $\mathbf{u}_0\in\mathbb{H}$ of $\mathbf{u}$ at $x^0$, that is, $\mathbf{u}_0(x)=\frac{f(0)}{2}\mathbf{e}(x^0)\max(x\cdot\nu(x^0),0)^2$ for some  $\nu(x^0)\in \partial B_1(0)\subset \mathbb{R}^n$ and $\mathbf{e}(x^0)\in \partial B_1(0)\subset \mathbb{R}^m$. Then there exist $r_0>0$ and $C<\infty$ such that
		$$
		\int_{\partial B_1(0)}\left|\frac{\mathbf{u}(x^0+rx)}{r^2}-\frac{f(0)}{2}\mathbf{e}(x_0)\max(x\cdot\nu(x^0),0)^2\right| d \mathcal{H}^{n-1}\leq Cr^{\frac{(n+2)\kappa}{2(1-\kappa)}},
		$$
		for every $x^0\in C_h$ and every $r\leq r_0<1$.
	\end{lemma}

	\begin{proof}
		First of all, we prove that such compact set $C_h$ is non-empty. Since the set $\Gamma_0(\mathbf{u})$ known from Definition \ref{regular} and Corollary \ref{coro4.5} is obviously non-empty, we only need to prove that $C_h$ is a compact set. Taking a sequence $\{x^i\}\in C_h$, and $\displaystyle\lim_{i\to \infty} x^i\to x^0$, we show that $x^0\in C_h$. It follows from
		$$\displaystyle\lim_{r\rightarrow 0} \frac{\mathbf{u}(x^i+rx)}{r^2}=\frac{f(0)}{2}\mathbf{e}(x^i)\max(x\cdot\nu(x^i),0)^2.$$
		Then
		\begin{equation*}
			\begin{aligned}
				&\displaystyle\lim_{r\rightarrow 0}\left| \frac{\mathbf{u}(x^0+rx)}{r^2}-\frac{f(0)}{2}\mathbf{e}(x^0)\max(x\cdot\nu(x^0),0)^2\right|\\
				\leq&\displaystyle\lim_{r\rightarrow 0}\Bigg[\left|\frac{\mathbf{u}(x^0+rx)}{r^2}-\frac{\mathbf{u}(x^i+rx)}{r^2}\right|+\left|\frac{\mathbf{u}(x^i+rx)}{r^2}-\frac{f(0)}{2}\mathbf{e}(x^i)\max(x\cdot\nu(x^i),0)^2\right|\\
				&+\left|\frac{f(0)}{2}\mathbf{e}(x^i)\max(x\cdot\nu(x^i),0)^2-\frac{f(0)}{2}\mathbf{e}(x^0)\max(x\cdot\nu(x^0),0)^2\right|\Bigg]\\
				=&0,
			\end{aligned}
		\end{equation*}
		where using the regularity of solution $u$ and $\displaystyle\lim_{i\to \infty} x^i\to x^0$. Therefore, we know that $C_h$ is obviously not empty.
		Next, the remaining proof consists of four parts.
		
		(i) Firstly, for any $x\in C_h$, it is easy to see that for any $\epsilon>0$, there exist $r_0(x, \epsilon)>0$ such that $W(\mathbf{u},x^0,r)\leq \epsilon +\frac{\alpha_n}{2}$ for any $r\in(0,r_0(x, \epsilon))$. Since $W(\mathbf{u},x, r)$ is  increasing with respect to $r$, then using Dini's theorem  there exists  a uniform $r_0$ independent of the choice of $x$ such that $$W(\mathbf{u},x,r)\leq \epsilon +\frac{\alpha_n}{2}\ \ \text{for any $r\in(0,r_0)$ and any $x\in C_h$.}$$
		
		(ii) Secondly, if $\rho_j\rightarrow 0$, $x^j \in C_h$ and $\mathbf{u}_j:=\mathbf{u}(x^j+\rho_j\cdot)/\rho_j^2\rightarrow \mathbf{v}$ in $W^{1,2}_{loc}(\mathbb{R}^n;\mathbb{R}^m)$ as $j\rightarrow \infty$, then  recalling that equality \eqref{eq3.7} again, $\mathbf{v}$ is a homogeneous global solution of degree 2 to the system \eqref{eq1.1} and
		\begin{equation*}
			\begin{aligned}
				M(\mathbf{v})=&\displaystyle\lim_{j\rightarrow \infty}  \frac{1}{(\rho\rho_j)^{n+2}} \int_{B_{\rho\rho_j}(x_j)}|\nabla \mathbf{u}(y)|^2+F'(0)|\mathbf{u}(y)|dy-\frac{2}{(\rho\rho_j)^{n+3}}\int_{\partial B_{\rho\rho_j}(x_j)}|\mathbf{u}(y)|^2d\mathcal{H}^{n-1}\\
				\leq &\displaystyle\lim_{j\rightarrow \infty}  \frac{1}{{(\rho\rho_j)}^{n+2}} \int_{{B_{\rho\rho_j}}(x_j)}|\nabla \mathbf{u}(y)|^2+F(|\mathbf{u}(y)|)dy-\frac{2}{(\rho\rho_j)^{n+3}}\int_{\partial B_{\rho\rho_j}(x_j)}|\mathbf{u}(y)|^2d\mathcal{H}^{n-1}\\
				=&\displaystyle\lim_{j\rightarrow \infty} W(\mathbf{u},x^j,\rho\rho_j) = \frac{\alpha_n}{2},
			\end{aligned}
		\end{equation*}
		which shows $ M(\mathbf{v})=\displaystyle\frac{\alpha_n}{2}.$ Due to Corollary \ref{coro4.5}, $\mathbf{v}\in \mathbb{H}.$
		
		(iii) We claim that for any $\rho>0$ small enough, $\mathbf{u}(x+\rho \cdot)/\rho^2$ is uniformly close to $\mathbb{H}$ in the $W^{1,2}_{loc}(\mathbb{R}^n;\mathbb{R}^m)\cap L^{\infty}_{loc}(\mathbb{R}^n;\mathbb{R}^m)$-topology for $x\in C_h$.
		
		To verify this claim, one may use the argument by contradiction. Assume this claim fails, then there exist  $\rho_j\rightarrow 0$ and $x^j\in C_h$ such that for any $\mathbf{h}\in \mathbb{H}$, there holds
		\begin{equation}\label{eq7.1}
			\|\mathbf{u}(x^j+\rho_j \cdot)/\rho_j^2-\mathbf{h}\|_{W^{1,2}_{B_1}(\mathbb{R}^n;\mathbb{R}^m)}+\|\mathbf{u}(x^j+\rho_j \cdot)/\rho_j^2-\mathbf{h}\|_{L^{\infty}_{B_1}(\mathbb{R}^n;\mathbb{R}^m)}\geq \delta >0.
		\end{equation} Let $\mathbf{u}_j:=\mathbf{u}(x^j+\rho_j \cdot)/\rho_j^2$. Owing to Theorem \ref{growth}, we obtain that $|\mathbf{u}_j|\leq C$, and  $|\nabla\mathbf{u}_j|\leq C$, then $\|\mathbf{u}_j\|_{L^{\infty}_{loc}(\mathbb{R}^n;\mathbb{R}^m)}\leq C$. Thus
		$\|\mathbf{u}_j\|_{W^{1,2}_{loc}(\mathbb{R}^n;\mathbb{R}^m)}\leq C$. Hence there is convergent subsequence, still denoted by $\mathbf{u}_j$,  such that $\mathbf{u}_j\rightarrow \mathbf{w}$ in $W^{1,2}_{loc}(\mathbb{R}^n;\mathbb{R}^m)\cap L^{\infty}_{loc}(\mathbb{R}^n;\mathbb{R}^m)$.  By (ii), we know that $\mathbf{w}\in \mathbb{H}$ ,  which contradicts with the fact \eqref{eq7.1}.
		
		(iv) Finally, we conclude the proof of Lemma \ref{regularity1}.
		
		In fact, the claim of (iii) implies that  the assumptions  in Proposition \ref{uniqueness} hold true and thus there is a constant  $C=C(n,\kappa)>0$ such that
		\begin{equation*}
			\begin{aligned}
				\int_{\partial B_1(0)}|\frac{\mathbf{u}(x_0+rx)}{r^2}-\mathbf{w}(x)| d\mathcal{H}^{n-1}&\leq C|W(\mathbf{u},x_0, r)-W(\mathbf{u},x^0,0+)|^{\frac{1}{2}}(\frac{r}{r_0})^{\frac{(n+2)\kappa}{2(1-\kappa)}}\\
				&\leq C(\epsilon+\frac{\alpha_n}{2}-\frac{\alpha_n}{2})^{\frac{1}{2}}(\frac{r}{r_0})^{\frac{(n+2)\kappa}{2(1-\kappa)}}\\
				&\leq C(n,\kappa)r^{\frac{(n+2)\kappa}{2(1-\kappa)}}.
			\end{aligned}
		\end{equation*}
		for any ${x}_0\in C_h$ and $r\in(0, r_0)$. Here $\mathbf{w}$ is the unique blow-up limit of $\mathbf{u}$ at $x_0$ and $\mathbf{w}\in \mathbb{H}$.
	\end{proof}

	\vspace{5pt}
	In the following, we use the Lemma \ref{regularity1} to prove the main result.
	
	\noindent{\it Proof of Theorem \ref{regularity}.}
	Consider $x^0\in \mathcal{R}_{\mathbf{u}}\subset \Gamma_0(\mathbf{u})$, by Lemma \ref{regularity1}, there exists $\delta_0>0$ such that $B_{2\delta_0}(x^0)\subset D$ and
	\begin{equation}\label{eq7.2}
		\int_{\partial B_1(0)}\left|\frac{\mathbf{u}(x^1+rx)}{r^2}-\frac{f(0)}{2}\mathbf{e}(x^1)\max(x\cdot\nu(x^1),0)^2 \right| d\mathcal{H}^{n-1}\leq Cr^{\frac{(n+2)\kappa}{2(1-\kappa)}},
	\end{equation}
	for every $x^1\in \mathcal{R}_{\mathbf{u}}\cap \overline{B_{\delta_0}(x^0)} \subset C_h$ and every $r\leq\min(\delta_0,r_0)\leq r_0<1$.
	
	In the following, we split the proof into the establishment of several claims.
	
	\textbf{Claim 1}. $x^1\longmapsto\nu(x^1)$ and $x^1\longmapsto\mathbf{e}(x^1)$ are H\"older-continuous with exponent $\beta$ on $\mathcal{R}_{\mathbf{u}}\cap \overline{B_{\delta_1}(x^0)}$ for some $\delta_1\in(0,\delta_0)$.
	
	\noindent{\it Proof of Claim 1}. In fact, the Proposition \ref{growth} leads to that
	\begin{equation*}
		\begin{aligned}
			&\frac{f(0)}{2}\int_{\partial B_1(0)}\left|\mathbf{e}(x^1)\max(x\cdot\nu(x^1),0)^2-\mathbf{e}(x^2)\max(x\cdot\nu(x^2),0)^2\right|d\mathcal{H}^{n-1}\\
			=&\int_{\partial B_1(0)}\Big{|}\frac{f(0)}{2}\mathbf{e}(x^1)\max(x\cdot\nu(x^1),0)^2-\frac{\mathbf{u}(x^1+rx)}{r^2}+\frac{\mathbf{u}(x^1+rx)}{r^2}-\frac{\mathbf{u}(x^2+rx)}{r^2}  \\
			&+\frac{\mathbf{u}(x^2+rx)}{r^2}  -\frac{f(0)}{2}\mathbf{e}(x^2)\max(x\cdot\nu(x^2),0)^2\Big{|}d\mathcal{H}^{n-1}\\
			\leq& 2Cr^{\frac{(n+2)\kappa}{2(1-\kappa)}} +\int_{\partial B_1(0)}\left|\frac{\mathbf{u}(x^1+rx)}{r^2}-\frac{\mathbf{u}(x^2+rx)}{r^2} \right|d\mathcal{H}^{n-1}\\
			\leq&2Cr^{\frac{(n+2)\kappa}{2(1-\kappa)}} +\int_{\partial B_1(0)}\int^1_0\left| \frac{\nabla\mathbf{u}(x^1+rx+t(x^2-x^1))}{r^2}\right|\left|x^1-x^2\right|dtd\mathcal{H}^{n-1}\\
			\leq&2Cr^{\frac{(n+2)\kappa}{2(1-\kappa)}}  +\int_{\partial B_1(0)}\int^1_0 \frac{C dist \left(x^1+rx+t(x^2-x^1),\Gamma_0(\mathbf{u})\right)}{r^2} \left|x^1-x^2\right|dtd\mathcal{H}^{n-1}.
		\end{aligned}
	\end{equation*}
	This implies,
	\begin{equation}\label{eq7.3-1}
		\begin{aligned}
			&\frac{f(0)}{2}\int_{\partial B_1(0)}\left|\mathbf{e}(x^1)\max\left(x\cdot\nu(x^1),0\right)^2-\mathbf{e}(x^2)\max\left(x\cdot\nu(x^2),0\right)^2\right|d\mathcal{H}^{n-1}\\
			\leq& 2Cr^{\frac{(n+2)\kappa}{2(1-\kappa)}}  +C_1\frac{\max\left(r,|x^1-x^2|\right)|x^1-x^2|}{r^2}\\
			\leq&2Cr^{\frac{(n+2)\kappa}{2(1-\kappa)}}+C_1r^{-1}|x^1-x^2|\\
			=&2Cr^{\frac{(n+2)\kappa}{2(1-\kappa)}}+C_1|x^1-x^2|^{1-\gamma}\\
			=&(2C+C_1)|x^1-x^2|^{\gamma {\frac{(n+2)\kappa}{2(1-\kappa)}}},
		\end{aligned}
	\end{equation}
	where we choose $\gamma:=\left(1+{\frac{(n+2)\kappa}{2(1-\kappa)}}\right)^{-1}$ and $r:=|x^2-x^1|^{\gamma}\leq \min (\delta_0,r_0)\leq r_0$.
	
	On the other hand, the left side of the inequality \eqref{eq7.3-1}  satisfies
	\begin{equation}\label{eq7.3}
		\begin{aligned}
			&\frac{f(0)}{4}\int_{\partial B_1(0)}\left|\mathbf{e}(x^1)\max(x\cdot\nu(x^1),0)^2-\mathbf{e}(x^2)\max(x\cdot\nu(x^2),0)^2\right|d\mathcal{H}^{n-1}\\
			\geq&c(n)\left(\left|\nu(x^1)-\nu(x^2)\right|+\left|\mathbf{e}(x^1)-\mathbf{e}(x^2)\right|\right).
		\end{aligned}
	\end{equation}
	Indeed, if not, then one may assume that there exist sequences of unit vectors  $\{\nu_j(x^1)\}$, $\{\nu_j(x^2)\}$, $\{\mathbf{e}_j(x^1)\}$  and $\{\mathbf{e}_j(x^2)\}$ such that
	$$T_j:=\frac{1}{c_j}\int_{\partial B_1(0)}|\mathbf{e}_j(x^1)\max(x\cdot\nu_j(x^1),0)^2-\mathbf{e}_j(x^2)\max(x\cdot\nu_j(x^2),0)^2|d\mathcal{H}^{n-1}\rightarrow 0 ,$$
	where  $c_j=| \nu_j(x^1) -\nu_j(x^2)|+|{\mathbf{e}}_j(x^1)-{\mathbf{e}}_j(x^2)|$.
	Due to the compactness,  for $j\rightarrow +\infty$, up to subsequences, $\nu_j(x^1)\rightarrow \bar{\nu}(x^1)$, $\nu_j(x^2)\rightarrow \bar{\nu}(x^2)$, $\mathbf{e}_j(x^1)\rightarrow \bar{\mathbf{e}}(x^1)$, $\mathbf{e}_j(x^2)\rightarrow \bar{\mathbf{e}}(x^2)$,
	$( \nu_j(x^1) -\nu_j(x^2))/c_j\rightarrow\eta$, and $(\mathbf{e}_j(x^1)-{\mathbf{e}}_j(x^2))/c_j\rightarrow \xi$. Since
	$$\int_{\partial B_1(0)}|\mathbf{e}_j(x^1)\max(x\cdot\nu_j(x^1),0)^2-\mathbf{e}_j(x^2)\max(x\cdot\nu_j(x^2),0)^2|d\mathcal{H}^{n-1}\rightarrow 0,$$
	then $\bar{\mathbf{e}}(x^1)=\bar{\mathbf{e}}(x^2)$ and $\bar{\nu}(x^1)=\bar{\nu}(x^2).$ It is easy to see that as $  j\to\infty$, we deduce
	\begin{eqnarray*}
		0&\gets& T_j \\
		&\geq&\frac{1}{c_j}\int_{S_j}\left|\left(\mathbf{e}_j(x^1)-\mathbf{e}_j(x^2)\right)\left(x\cdot\nu_j(x^1)\right)^2+\mathbf{e}_j(x^2)\left(x\cdot\nu_j(x^1)\right)^2 -\mathbf{e}_j(x^2)\left(x\cdot\nu_j(x^2)\right)^2\right|d\mathcal{H}^{n-1}\\
		&\to& \int_{\partial B_1(0)\cap \{x\cdot\bar{\nu}(x^1)>0\}}\left|\xi \left(x\cdot\bar{\nu}(x^1)\right)^2+2\bar{\mathbf{e}}(x^1)x\cdot\bar{\nu}(x^1)x\cdot\eta\right|d\mathcal{H}^{n-1}
	\end{eqnarray*} where $S_j=\partial B_1(0)\cap \{x\cdot\nu_j(x^1)>0\}\cap\{x\cdot\nu_j(x^2)>0\}$.
	Thus,  $\xi=\displaystyle\frac{2\bar{\mathbf{e}}(x^1)x\cdot\eta}{ x\cdot\bar{\nu}(x^1)}$. However, one also has
	$$0=\left(|\mathbf{e}_j(x^1)|^2-|\mathbf{e}_j(x^2)|^2\right)/c_j=\left(\left(\mathbf{e}_j(x^1)+\mathbf{e}_j(x^2)\right)\left(\mathbf{e}_j(x^1)-\mathbf{e}_j(x^2)\right)\right)/c_j\to 2\bar{\mathbf{e}}(x^1)\cdot \xi,$$
	as  $  j\to\infty$ and then
	$$0=\bar{\mathbf{e}}(x^1)\cdot \xi=\bar{\mathbf{e}}(x^1)\cdot\frac{2\bar{\mathbf{e}}(x^1)x\cdot\eta}{ x\cdot\bar{\nu}(x^1)}=\frac{2x\cdot\eta}{ x\cdot\bar{\nu}(x^1)}, $$
	which arrives at a contradiction. Therefore, we finish the proof of Claim 1.
	
	\textbf{Claim 2}. $\forall \ \epsilon>0$, there exists $\delta_2\in (0,\delta_1)$ such that for any  $x^1\in \mathcal{R}_{\mathbf{u}}\cap \overline{B_{\delta_0}(x^0)}$,  we obtain
	
	\begin{equation}\label{eq7.4}
		\begin{split}
			|\mathbf{u}(y)|&=0 \quad \text{for}\quad y\in  \overline{B_{\delta_2}(x^1)} \quad \text{satisfying}\quad (y-x^1)\cdot \nu(x^1)<-\epsilon|y-x^1|, \\
			\text{and}\qquad\qquad&\\
			|\mathbf{u}(y)|&>0 \quad \text{for}\quad y\in  \overline{B_{\delta_2}(x^1)} \quad \text{satisfying}\quad (y-x^1)\cdot \nu(x^1)>\epsilon|y-x^1|.
		\end{split}
	\end{equation}
	
	\noindent{\it Proof of Claim 2}.  One may assume \eqref{eq7.4} does not hold true, and then there are a sequence $\mathcal{R}_{\mathbf{u}}\cap \overline{B_{\delta_0}(x^0)}\ni x^m \to \bar{x}$ and a sequence $y^m-x^m\to 0$ as $m\to\infty$ such that
	\begin{eqnarray}\label{eq7.5}
		\begin{aligned}
			\text{either}\ |\mathbf{u}(y^m)|&>&0 \quad \text{for}\quad y^m\in  \overline{B_{\delta_2}(x^m)} \quad \text{satisfying}\quad (y^m-x^m)\cdot \nu(x^m)&<&&-\epsilon|y^m-x^m|,& \\
			\text{or}\quad |\mathbf{u}(y^m)|&=&0 \quad \text{for}\quad y^m\in  \overline{B_{\delta_2}(x^m)} \quad \text{satisfying}\quad (y^m-x^m)\cdot \nu(x^m)&>&&\epsilon|y^m-x^m|.&
		\end{aligned}
	\end{eqnarray}
	The inequality \eqref{eq7.2} leads to that
	$\mathbf{u}_j(x)=\frac{\mathbf{u}(x^j+|y^j-x^j|x)}{|y^j-x^j|^2}$ is converges $\frac{f(0)}{2}\mathbf{e}(\bar{x})\max(\bar{x}\cdot\nu(\bar{x}),0)^2$ as $j\to \infty$ in $C^{1,\alpha}_{loc}(\mathbb{R}^n;\mathbb{R}^m)$ and that $\mathbf{u}_j=\mathbf{0}$ on each compact subset $\Omega$ of $\{x\cdot \nu(\bar{x})<0\}$ provided that $j$ is large enough. Therefore, this is contradictory to \eqref{eq7.5} for large enough $j$.
	
	\textbf{Claim 3}. There exists $\delta_3\in(0,\delta_2)$ such that $\partial \{ |\mathbf{u}|>0\}$ is in $B_{\delta_3}(x^0)$ the graph of a differentiable function.
	
	$Proof \ of \ Claim \ 3$.  let $\nu(x^0):=\mathbf{e}^n:=(0,...,0,1)$, $\mathbf{e}(x^0):=\mathbf{e}^1:=(1,0,...0)$ and fixing $\epsilon:=\frac{1}{2}$, then there exists a constant $\delta_2>0$ (see Claim 2) such that \eqref{eq7.4} holds true. One may define two functions  $g^{\pm}:  B'_{\delta_2/2}(0)\longmapsto [-\infty,+\infty]$ as follows:
	\begin{eqnarray*}
		g^+(x') := \sup\{x_n:x^0+(x',x_n)\in \Gamma(\mathbf{u})\}\ \text{and} \ g^-(x')\ := \inf\{x_n:x^0+(x',x_n)\in \Gamma(\mathbf{u})\}.
	\end{eqnarray*}
	For $g^{\pm}$, we have the following properties.
	
	$\bullet$ \ \ $-\infty<g^-\leq g^+<+\infty$.  In fact, for $y\in \overline{B_{\delta_2}(x^0)}\cap\Gamma(\mathbf{u})$, applying the Claim 2 for $x^0$, it follows that
	$$-\epsilon |y-x^0|\leq(y-x^0)\cdot\nu(x^0)\leq \epsilon|y-x^0|,$$
	which shows $-\epsilon\delta_2\leq(y-x^0)_n \leq \epsilon\delta_2,$
	i.e. $-\epsilon\delta_2\leq x_n \leq \epsilon\delta_2 ,$
	which implies $-\infty<g^-\leq g^+<+\infty$ on $B'_{{\delta_2}/{2}}(0)$.
	
	$\bullet$ \ \ $g^+(x')=g^-(x')$ for any $ x'\in B'_{\delta_2/2}(0)$. In fact, since $\left|\nabla\mathbf{u}\left(x^0+(x',g^-(x'))\right)\right|=0$ for every $x'\in B'_{\delta_2/2}(0)$, then  $x^0+(x',g^-(x'))\in\Gamma_0(\mathbf{u})$. Recalling Corollary \ref{coro5.8} and $x^0\in \mathcal{R}_{\mathbf{u}}$, it infers  that $x^0+(x',g^-(x'))\in \mathcal{R}_{\mathbf{u}}$ for any $x'\in B'_{\delta_2/2}(0)$. Notice that the facts  \eqref{eq7.2}-\eqref{eq7.4}, there exists a $\delta_3\in(0, \delta_2)$ such that  $|\nu(x)-\nu(y)|\leq \frac{1}{4}$ when $x, \  y\in \overline{B_{\delta_3}(x^0)}\cap\Gamma(\mathbf{u})$.
	According to $x^1:= x^0+(x',g^-(x'))\in\mathcal{R}_{\mathbf{u}}\cap \overline{B_{\delta_1}(x_0)}$ , Applying Claim 2 for $x'$, it follows that
	$$
	\begin{aligned}
		&|g^+(x')-g^-(x')|-\frac{1}{4}|g^+(x')-g^-(x')|\\
		&\leq (0, g^+(x')-g^-(x')) \cdot (\nu(x^0)+ \nu(x^1)-\nu(x^0))\\
		&\leq\frac{1}{2}|g^+(x')-g^-(x')|
	\end{aligned}
	$$ which shows $g^+(x')=g^-(x')$ in $B'_{\delta2/2}(0)$.
	
	$\bullet$ \ \ $g(x'):=g^+(x')=g^-(x')\in C^{0,1}(\overline{B'_{\delta_3}(0)})$. Indeed, noting the Claim 2, one may see that there  is in a uniform cone such that this cone  stay above $\partial\{|\mathbf{u}|>0\}\cap B_{\delta_3}(x^0)$, which shows $g\in C^{0,1}(\overline{B'_{\delta_3}(0)})$. In particular, all free boundary points close to $x^0$ belong to $\mathcal{R}_{\mathbf{u}}$, i.e. there are no other free boundary points (for example free boundary points with non-vanishing gradient) in the neighbourhood of $x^0$.
	
	$\bullet $ \ \ $g\in C^{1, \beta}(\overline{B'_{\delta_3}(0)})$. In fact,  for some  sufficiently small $\delta_3\in(0,\delta_2)$, $\partial \{ |\mathbf{u}|>0\}$ is the graph of a Lipschitz function on $\overline{B_{\delta_3}(x^0)}$. From Claim 1, it follows that there exists a uniform constant $C>0$ such that $|\nu(x)-\nu(y)|\leq  C|x-y|^{\beta}$ for any $x, y\in  \overline{B_{\delta_3}(x^0)}$, which concludes the proof of our main result.
	
	$\hfill\square$

	\begin{corollary}{(Macroscopic criterion for regularity)}
		Let $\bar{\alpha}_n $ be the constant defined in Corollary \ref{coro4.5}. Then $B_{2r}(x^0) \subset D, x^0 \in \{|{u}| > 0\}$ and $W(\mathbf{u},x^0,r) <\bar{\alpha}_n$ imply that $\partial\{|\mathbf{u}| > 0\}$ is in an open neighbourhood of $x^0$ a $C^{1,\beta}$-surface.
	\end{corollary}
	
	\begin{proof}  Due to $C^{1,\beta}$-regularity of $\mathbf{u}$ and Theorem \ref{regularity}, it suffices to show that for $W(\mathbf{u},x^0,r) <\bar{\alpha}_n$ either (i) $\nabla\mathbf{u}(x^0\neq 0) $or (ii) $x^0\in\mathcal{R}_{\mathbf{u}}$. If both (i) and (ii) fail then by Lemma \ref{monotonicity}, Lemma \ref{property} and Corollary \ref{coro4.5}, $W(\mathbf{u},x^0,r)\geq\displaystyle\lim_{r\rightarrow 0} W(\mathbf{u},x^0,r)\geq    \bar{\alpha}_n$ contradicting the assumption.

	\end{proof}

	\vspace{10pt}
	\appendix
	\renewcommand{\appendixname}{Appendix~\Alph{section}}
	\section{Regularity and uniqueness of the solution to the system \eqref{eq1.1}.}
	We introduce the corresponding elliptic system for this minimization problem, and then we provide proofs of the regularity (Proposition \ref{rem1.2}) and uniqueness of the solution to this system.
	
	Since the energy of the minimization problem \eqref{eq1.0} has non-negativity convexity and weak convergence that has lower semicontinuity, we can know that $E(\mathbf{u})=\int_{D}{|\nabla\mathbf{u}|}^2+F(|\mathbf{u}|)dx$ exists a minimizer for each $\mathbf{g}\in W^{1,2}(D;\mathbb{R}^m )$.
	Then for each $\boldsymbol{\phi} \in W^{1,2}_0(D;\mathbb{R}^m)$,
	\begin{align}\nonumber
		E(\mathbf{u})& \leq E(\mathbf{u}+ \epsilon \boldsymbol{\phi})\\
		&=\int_D|\nabla\mathbf{u}+ \epsilon \nabla \boldsymbol{\phi}|^2+F(|\mathbf{u}+ \epsilon \boldsymbol{\phi} |)dx.\nonumber
	\end{align}
	Simplifying the above inequality directly, we obtain
	\begin{equation}\label{eq2.4}
		0\leq \epsilon\int_D 2\nabla\mathbf{u}\cdot \nabla\boldsymbol{\phi}dx+\epsilon^2\int_D|\nabla\boldsymbol{\phi}|^2dx+\int_D F(|\mathbf{u}+ \epsilon\boldsymbol{\phi}|)-F|\mathbf{u}|dx.
	\end{equation}
	The  convexity of $F$ implies that
	$$
	F(|\mathbf{u}+\epsilon \boldsymbol{\phi}|)-F(|\mathbf{u}|)\leq F'(|\mathbf{u}+\epsilon \boldsymbol{\phi}|)(|\mathbf{u}+\epsilon \boldsymbol{\phi}|-|\mathbf{u}|)\leq F'(|\mathbf{u}+\epsilon \boldsymbol{\phi}|)|\epsilon| | \boldsymbol{\phi}|.
	$$
	
	In the one case $\{\mathbf{u}=\mathbf{0}\}$. From the assumption of $F$, it implies that $F'(|\mathbf{u}+\epsilon \boldsymbol{\phi}|)\leq C_0$, where $C_0$ is a given positive constant. Hence, the inequality (\ref{eq2.4}) gives that
	$$
	2\left|\int_D\nabla\mathbf{u}\cdot\nabla\boldsymbol{\phi}dx\right|\leq C_0\|\boldsymbol{\phi}\|_{L^1(D;\mathbb{R}^m)},
	$$
	therefore, near the free boundary we have the fact $\Delta\mathbf{u}\in L^{\infty}(D;\mathbb{R}^m)$. Then, $L^p-$ and $C^\alpha-$ theory give that $\mathbf{u}\in W^{2,p}_{loc} (D;\mathbb{R}^m)\cap C^{1,\alpha}_{loc}(D;\mathbb{R}^m)$ for $p\in[1,+\infty)$, $\alpha\in (0,1)$. Thus, we can define strong solution in $\{\mathbf{u}=\mathbf{0}\}$.
	i.e. $\Delta\mathbf{u}=\mathbf{0}$ a.e. in $\{\mathbf{u}=\mathbf{0}\}$.
	
	In another case $\{|\mathbf{u}|>\delta>0\}$. Together with the assumption \eqref{eq1.4} and the Taylor expansion of $F$, we have
	\begin{equation}\label{eq2.5}
		0\leq  \epsilon\int_D 2\nabla\mathbf{u}\cdot \nabla\boldsymbol{\phi}dx+\epsilon^2\int_D|\nabla\boldsymbol{\phi}|^2dx+ \int_D F'(|\mathbf{u}|)\cdot \frac{\mathbf{u}}{|\mathbf{u}|}\cdot \epsilon \boldsymbol{\phi}+ o(\epsilon) dx.
	\end{equation}
	Dividing by $\epsilon$ and letting $\epsilon\rightarrow 0$, then $\mathbf{u}$ satisfies that
	\begin{equation}\nonumber
		\Delta\mathbf{u}=F'(|\mathbf{u}|)\frac{\mathbf{u}}{2|\mathbf{u}|}=f(|\mathbf{u}|)\frac{\mathbf{u}}{|\mathbf{u}|}  \quad \text{in}\quad \{|\mathbf{u}|>\delta>0\},
	\end{equation}
	where $f(s)=\frac{1}{2}F'(s)$. Therefore, combining the above two situations, it gives that
	\begin{equation*}
		\Delta\mathbf{u}=f(|\mathbf{u}|)\frac{\mathbf{u}}{|\mathbf{u}|} \chi_{\{|\mathbf{u}|>0\}}  \quad \text{in}\quad D.
	\end{equation*}

	Now, we shall prove the uniqueness of solution to equation \eqref{eq1.1}. Assume there exist different solutions $\mathbf{v},\mathbf{u} \in W^{1,2}(D;\mathbb{R}^m)$ with the same boundary data $\mathbf{u}_0$ such that
	\begin{equation}\nonumber
		\begin{cases}
			\Delta\mathbf{u}=f(|\mathbf{v}|)\displaystyle\frac{\mathbf{v}}{|\mathbf{v}|} \chi_{\{|\mathbf{v}|>0\}}  \quad \text{in}\quad D,\\
			\Delta\mathbf{u}=f(|\mathbf{u}|)\displaystyle\frac{\mathbf{u}}{|\mathbf{u}|} \chi_{\{|\mathbf{u}|>0\}}  \quad \text{in}\quad D.
		\end{cases}
	\end{equation}
	Then integrating by parts, it yields that $\mathbf{v}$ and $\mathbf{u}$ satisfy the following weak equation
	\begin{equation}\nonumber
		\begin{cases}
			\int_D \nabla\mathbf{v}\cdot \nabla\boldsymbol{\phi}+ f(|\mathbf{v}|) \displaystyle\frac{\mathbf{v}}{|\mathbf{v}|}\chi_{\{|\mathbf{v}|>0\}}\cdot \boldsymbol{\phi} dx=0,\\
			\int_D \nabla\mathbf{u}\cdot \nabla\boldsymbol{\phi}+ f(|\mathbf{u}|) \displaystyle\frac{\mathbf{u}}{|\mathbf{u}|}\chi_{\{|\mathbf{u}|>0\}}\cdot \boldsymbol{\phi} dx=0.
		\end{cases}
	\end{equation}
	Subtract these two equations, and let $\boldsymbol{\phi}=\mathbf{v}-\mathbf{u}$ we can see that
	\begin{equation*}
		\begin{aligned}
			0\leq& \int_D \left(\nabla(\mathbf{v}-\mathbf{u})\right)^2 dx\\
			=&-\int_D\big(f(|\mathbf{u}|) \frac{\mathbf{u}}{|\mathbf{u}|}\chi_{\{|\mathbf{u}|>0\}}- f(|\mathbf{v}|) \frac{\mathbf{v}}{|\mathbf{v}|}\chi_{\{|\mathbf{v}|>0\}}\big) \cdot (\mathbf{u}-\mathbf{v}) dx \\
			=&-\int_{D\cap\{|\mathbf{u}\neq 0|\}\cap\{|\mathbf{v}\neq 0|\}}\big(f(|\mathbf{u}|) \frac{\mathbf{u}}{|\mathbf{u}|}\chi_{\{|\mathbf{u}|>0\}}- f(|\mathbf{v}|) \frac{\mathbf{v}}{|\mathbf{v}|}\chi_{\{|\mathbf{v}|>0\}}\big) \cdot (\mathbf{u}-\mathbf{v})dx\\
			&-\int_{D\cap\left\{\{|\mathbf{u}= 0|\}\cup\{|\mathbf{v}= 0|\}\right\}}\big(f(|\mathbf{u}|) \frac{\mathbf{u}}{|\mathbf{u}|}\chi_{\{|\mathbf{u}|>0\}}- f(|\mathbf{v}|) \frac{\mathbf{v}}{|\mathbf{v}|}\chi_{\{|\mathbf{v}|>0\}}\big) \cdot (\mathbf{u}-\mathbf{v})dx\\
			\leq&-\int_{D\cap\{|\mathbf{u}\neq 0|\}\cap\{|\mathbf{v}\neq 0|\}}\big(f(|\mathbf{u}|) \frac{\mathbf{u}}{|\mathbf{u}|}\chi_{\{|\mathbf{u}|>0\}}- f(|\mathbf{v}|) \frac{\mathbf{v}}{|\mathbf{v}|}\chi_{\{|\mathbf{v}|>0\}}\big) \cdot (\mathbf{u}-\mathbf{v})dx,
		\end{aligned}
	\end{equation*}
	and a direct computation gives that in $D\cap\{|\mathbf{u}\neq 0|\}\cap\{|\mathbf{v}\neq 0|\}$,
	\begin{equation*}
		\begin{aligned}
			&\big(f(|\mathbf{u}|) \frac{\mathbf{u}}{|\mathbf{u}|}- f(|\mathbf{v}|) \frac{\mathbf{v}}{|\mathbf{v}|}\big) \cdot (\mathbf{u}-\mathbf{v})\\
			=&f(|\mathbf{u}|)|\mathbf{u}| +f(|\mathbf{v}|) |\mathbf{v}|-\frac{f(|\mathbf{u}|)|\mathbf{v}|+f(|\mathbf{v}|)|\mathbf{u}|}{|\mathbf{u}||\mathbf{v}|}\mathbf{u}\cdot \mathbf{v}\\
			\leq&f(|\mathbf{u}|)|\mathbf{u}| +f(|\mathbf{v}|) |\mathbf{v}|-\frac{f(|\mathbf{u}|)|\mathbf{v}|+f(|\mathbf{v}|)|\mathbf{u}|}{|\mathbf{u}||\mathbf{v}|}|\mathbf{u}\cdot \mathbf{v}|\\
			\leq&f(|\mathbf{u}|) |\mathbf{u}|+f(|\mathbf{v}|) |\mathbf{v}|-f(|\mathbf{u}|) |\mathbf{v}|- f(|\mathbf{v}|) |\mathbf{u}|.\\
		\end{aligned}
	\end{equation*}
	Therefore, it shows that
	\begin{equation*}
		\begin{aligned}
			0\leq& \int_D \left(\nabla(\mathbf{v}-\mathbf{u})\right)^2dx\\
			\leq&-\int_{D\cap\{|\mathbf{u}\neq 0|\}\cap\{|\mathbf{v}\neq 0|\}} f(|\mathbf{u}|) |\mathbf{u}|+f(|\mathbf{v}|) |\mathbf{v}|-f(|\mathbf{u}|) |\mathbf{v}|- f(|\mathbf{v}|) |\mathbf{u}| dx\\
			=&-\int_{D\cap\{|\mathbf{u}\neq 0|\}\cap\{|\mathbf{v}\neq 0|\}} f(|\mathbf{u}|)(|\mathbf{u}|-|\mathbf{v}|)- f(|\mathbf{v}|)(|\mathbf{u}|-|\mathbf{v}|)dx\\
			=&-\int_{D\cap\{|\mathbf{u}\neq 0|\}\cap\{|\mathbf{v}\neq 0|\}}(|\mathbf{u}|-|\mathbf{v}|)\left(f(|\mathbf{u}|)-f(|\mathbf{v}|)\right)dx\\
			\leq& 0,
		\end{aligned}
	\end{equation*}
	where using the convexity of $F$, we arrive at $(|\mathbf{u}|-|\mathbf{v}|)(f(|\mathbf{u}|)-f(|\mathbf{v}|))\geq 0$. Consequently, we infer that $\mathbf{u}=\mathbf{v}$, i.e. we obtain the uniqueness of solution to the system \eqref{eq1.1}.
	\vspace{5pt}

	\section{Properties of the Laplace-Beltrami operator }
	We will provide some properties of the Laplace-Beltrami operator on the unit sphere in $\mathbb{R}^n$, which help us understand the property (Proposition \ref{prop4.1}) of global solutions of this system.
	
	\begin{lemma}\label{laplace}
		Let $\Delta'$ be the Laplace-Beltrami operator on the unit sphere in $\mathbb{R}^n$, let the domain $\Omega' \subset \partial B_1(0)\subset \mathbb{R}^n$, let $\mathcal{L}:=-\Delta'+q $ where $q\in C^0(\Omega')$ such that $q \geq q_0>0$ in $\Omega'$, and let $\lambda_k(\mathcal{L},\Omega')$ denote the $k-th$ eigenvalue with respect to the eigenvalue problem
		\begin{equation}
			\begin{aligned}
				\mathcal{L}v&=\lambda v \quad \text{in} \quad \Omega'\\
				v&=0 \quad \text{on} \quad \partial \Omega'
			\end{aligned}\nonumber
		\end{equation}
		here $\partial \Omega' $ denotes the boundary of $\Omega' $ relative to $\partial B_1$.
		\begin{enumerate}
			\item If $\bar{\Omega}'\subset\Omega'$ then $\lambda_k(\mathcal{L},\bar{\Omega}')\geq \lambda_k(\mathcal{L},\Omega')$ for every $k\in \mathbb{N}.$ For $k=1$ the inequality is strict.
			\item $\lambda_k(\mathcal{L},\Omega')\geq q_0 + \lambda_k(-\Delta', \Omega')$ for every $k\in \mathbb{N}$; in case $q\neq q_0$ the inequality becomes a strict inequality.
			\item $q=\displaystyle\frac{1}{2h} f'(0)$ and $\Omega'\subset \partial B_1 \cap \{x_n>0\}$ and $v\in W^{1,2}(\partial B_1)$ satisfies $$\mathcal{L}v=2nv \quad \text{in} \quad {\Omega'},$$
			which implies $v=ah$ for some real number $a\neq 0$. Here $h(x)=\frac{1}{2}\max(x_n,0)^2$
			\item $\Omega'\subset \partial B_1(0) \cap \{x_n>-\delta(n,q_0)\}$ and $v\in W^{1,2}(\partial B_1(0))$ satisfies $$\mathcal{L}v=2nv \quad \text{in} \quad {\Omega'}, $$
			that there exists a function $$f_{\Omega'}\in\{f | \ \mathcal{L} f=\lambda_1(\mathcal{L}, \Omega') f \  \text{in}\ \Omega', \ \text{and}\  f>0\},$$ where $\lambda_1(\mathcal{L}, \Omega')$ is the first eigenvalue, and $f_{\Omega'}$ depending only on $\Omega'$, such that $v=af_{\Omega'}$ for a real number $a\neq 0$. 
		\end{enumerate}
	\end{lemma}
	\begin{remark}
		For $q=\frac{1}{h} f(0)$, the eigenvalue problem on the sphere becomes
		\begin{equation}
			\begin{aligned}
				-\Delta'v +\displaystyle\frac{2}{\cos^2\theta}v&=2nv \quad &\text{in}&\quad \Omega'\subset \{x_n>0\},\\
				v&=0 \quad &\text{on}& \quad \partial \Omega'.
			\end{aligned}\nonumber
		\end{equation}
		Indeed, $q=\displaystyle\frac{1}{h} f(0)=f(0)\cdot \displaystyle\frac{2}{f(0)x_n^2}=\displaystyle\frac{2}{x_n^2} =\displaystyle\frac{2}{\cos^2\theta}$.
		Thus \cite[Appendix]{asu} can be further extended to the related this $q$.
	\end{remark}

	\section*{Conflicts of interest/Competing interests:}
	The authors declare that they have no conflict of interest/Competing interest.

\end{document}